\documentclass[letterpaper,11pt,reqno]{amsart} 
\usepackage[portrait,margin=1in]{geometry} 

\usepackage{mathrsfs,xfrac} 
\usepackage[colorlinks=true,linkcolor=blue,citecolor=blue,urlcolor=blue]{hyperref} 
\usepackage{amsmath,amssymb,amsthm,amsfonts,amsbsy,latexsym,dsfont,color,graphicx,enumitem}

\usepackage{caption}
\usepackage{regexpatch}
\usepackage{todonotes}
\usepackage{comment}

\makeatletter
\xpatchcmd{\@todo}{\setkeys{todonotes}{#1}}{\setkeys{todonotes}{inline,#1}}{}{}
\makeatother

\newtheorem{thm}{Theorem}[section]
\newtheorem{lem}[thm]{Lemma}
\newtheorem{cor}[thm]{Corollary}

\newtheorem{conj}[thm]{Conjecture}

\theoremstyle{definition}
\newtheorem{defn}[thm]{Definition}
\newtheorem{obs}[thm]{Observation}
\newtheorem{rem}[thm]{Remark}
\newtheorem{ex}[thm]{Example}
\newtheorem{ques}[thm]{Question} 
\renewcommand{\le}{\leqslant}  
\renewcommand{\ge}{\geqslant}

\newcommand{\wt}{\widetilde}

\newcommand{\ind}{\mathds{1}}
\newcommand{\eps}{\varepsilon}

\newcommand{\abs}[1]{\left\vert#1\right\vert}

\newcommand{\ie}{{i.e.,}}

\newcommand{\eg}{{e.g.,}}

\let\ga=\alpha    
           \let\go=\omega     
  
\let\gC=\Gamma

\newcommand{\cF}{\mathcal{F}}
\newcommand{\cG}{\mathcal{G}}
\newcommand{\cL}{\mathcal{L}}

\newcommand{\cR}{\mathcal{R}}

\newcommand{\bN}{\mathbb{N}}
\newcommand{\bR}{\mathbb{R}}
\newcommand{\bT}{\mathbb{T}}

\newcommand{\bZ}{\mathbb{Z}}        

\DeclareMathOperator{\pr}{\mathds{P}}

\DeclareMathOperator{\DL}{DL}
\DeclareMathOperator{\Aut}{Aut} 
\newcommand{\wh}[1]{\widehat{#1}}

\newcommand{\fmax}{\mathrm{FMaxSF}}

\newcommand{\FMSF}{\mathrm{FMSF}}
\newcommand{\w}{\mathbf{w}}
\newcommand{\level}{\mathrm{Level}}
\newcommand{\dist}{\mathrm{dist}}

\begin{document}
\title{Weighted-amenability and percolation}

\author[Terlov]{Grigory Terlov}
\address{
Grigory Terlov \\
Department of Statistics and Operations Research,
University of North Carolina,
Chapel Hill, NC,
USA,
and
HUN-REN Alfr\'ed R\'enyi Institute of Mathematics,
Re\'altanoda u. 13-15, Budapest 1053, Hungary.
}
\email{gterlov[at]unc.edu}

\author[Tim\'ar]{\'Ad\'am Tim\'{a}r}
\address{
\'Ad\'am Tim\'{a}r\\
Division of Mathematics,
The Science Institute, University of Iceland
Dunhaga 3 IS-107 Reykjavik, Iceland and
HUN-REN Alfr\'ed R\'enyi Institute of Mathematics,
Re\'altanoda u. 13-15, Budapest 1053, Hungary.
}
\email{madaramit[at]gmail.com}


\begin{abstract}
In 1999, Benjamini, Lyons, Peres, and Schramm introduced a notion of weighted-amenability for transitive graphs that is equivalent to the amenability of its automorphism group. 
For unimodular graphs this notion coincides with classical graph-amenability and has been intensely studied. In the present work, we show that many classical unimodular results can be extended to the nonunimodular setting, which is further motivated by recent progress in the mcp (measure class preserving or quasi-pmp) setting of measured group theory. 
To this end, we prove new characterizations of weighted-amenability, in particular that it is equivalent to all finite unions of levels inducing amenable graphs. Hutchcroft conjectured that the latter property implies that $p_h<p_u$, where $p_h$ is the critical probability for the regime where clusters of Bernoulli percolation are infinite total weight and $p_u$ is the uniqueness threshold. We prove a relaxed version of his conjecture à la Pak--Smirnova-Nagnibeda. Further characterizations are given in terms of the spectral radius and invariant spanning forests. One of the consequences is the continuity of the phase transition at $p_h$ for weighted-nonamenable graphs.
\end{abstract}
\maketitle
\setcounter{tocdepth}{1}
\tableofcontents

\section{Introduction and main results}

Percolation theory began in 1957 by physicists \cite{broadbent1957percolation}, who were interested in studying the percolation of water in porous stones. Naturally, they considered 
a simplified version of the real-world phenomena by equating the stone with $\bZ^d$ grid and letting the passages for water (\ie\ bonds or edges) be open at random with probability $p$ independently of each other. Today it is a classical model and is known as \textit{Bernoulli$(p)$ percolation}. 
By switching the focus from lattices and trees, in 1996 Benjamini and Schramm opened up the way to study percolation on general \textit{connected locally finite quasi-transitive} graphs \cite{BSbeyond}. In turn, quasi-transitive graphs were further split into two categories: unimodular (those that ``look like" Cayley graphs) and nonunimodular \cite{Haggstrom99,BLPS99inv,Timar06nonu}.

For the class of nonunimodular transitive graphs, many standard techniques fail because of an intrinsic asymmetry. Thus, besides a few notable exceptions, many conjectures stated for all transitive graphs have been resolved only under the additional assumption of unimodularity. Transitive graphs are naturally equipped with \textit{a relative weight function}, which is constant $1$ exactly for unimodular transitive graphs. Decomposition of vertices in sets of different weights quantifies the asymmetry in case of nonunimodular graphs. Besides posing a challenge from a technical perspective, another source of interest in this class of graphs is the particular phase transition phenomena that they produce \cite{PeresPeteScol,Timar06nonu,Pengfei18, Hutchcroft20}. Namely, due to the weights, the infinite subgraphs of nonunimodular graphs could be of two types: where the total weight is infinite or finite. We call such subgraphs \textit{heavy} or \textit{light}, respectively. Light infinite subgraphs often have properties somewhat in between finite and infinite components of the ``usual" transitive graphs, thus exhibiting particularly interesting behavior. While often being more mysterious than their unimodular nonamenable counterparts, in some regards nonunimodular graphs are better understood. A breakthrough result by Hutchcroft \cite{Hutchcroft20} shows the existence of several critical exponents for nonunimodular transitive graphs, and that they take mean-field values. This makes nonunimodular transitive graphs one of the few classes of transitive graphs where explicit critical exponents are known. 
It is important to note that even for unimodular graphs the existence of a nonunimodular subgroup of its automorphisms is useful, and in fact the mentioned results were proved in this generality. We will also follow this more general setup.
Finally, another source of interest comes from measured group theory and measurable combinatorics, where an increasing amount of attention has been paid to the \textit{measure class preserving} (mcp also known as quasi-pmp) graphs, equivalence relations, and actions of countable groups \cite{Ts:hyperfinite_ergodic_subgraph,AnushRobin,FmaxSF22, bowen2024one}. Establishing results in the mcp setting is in many ways parallel to studying percolation on nonunimodular transitive graphs through the relative weight function which is also defined for the orbit equivalence relation of such actions (and is known as the Radon--Nikodym cocycle). The connection between these fields extends beyond a similarity in themes and techniques, as results in one of them often enable advances in the other. The rising interest in the mcp setting is largely due to the success of the program initiated by Tserunyan in 2017 that aimed at filling the gaps in the understanding of such equivalence relations in comparison to their pmp counterpart. One of the highlights of this program is the recently announced resolution of a long standing question of measure equivalence classification of Baumslag-Solitar groups \cite{MEBaumslagSolitar}. The present work develops the theory on the percolation side of the interplay and fits into the same program.


For the above reasons, we find it timely to revisit the notion of {\it weighted-amenability} of a graph and study its fundamental properties (see Definition~\ref{def:wamen}). This notion was first introduced in \cite{BLPS99inv}, where the authors showed its equivalence to the amenability of the automorphism group of the graph. 
For unimodular transitive graphs (where the weight function is constant $1$) it reduces to the classical notion of graph-amenability.

\subsection{Percolation theory terminology}\label{sec:percintro}
We now introduce some of the common terminology in the field and our conventions. A more detailed overview of preliminaries can be found in Section~\ref{sec:preq}. As above, a \textbf{bond percolation} process on a graph $G=(V,E)$ is a probability measure $\textbf{P}$ on $2^E$. Throughout the paper we will assume that $G$ is connected and locally finite.
We refer to elements $\omega \in 2^E$ as \textbf{configurations} and we say that an edge $e \in E$ is \textit{present} (or \textit{open}) in $\omega$ if $e \in \omega$. 
The connected components of $\omega$ are called \textbf{clusters}. 
For $p \in [0,1]$, a bond percolation process on $G$ is called $\mathrm{Bernoulli}(p)$ if every edge is present in a configuration independently with probability $p$. 
We denote the measure associated with $\mathrm{Bernoulli}(p)$ bond percolation by $\pr_p$ and drop the word ``bond'' when it is clear from the context. A \textbf{site percolation} is defined similarly but as a measure on $2^V$.

Let $\Gamma$ be a closed subgroup of $\Aut(G)$. For $x,y \in V$ we define \textbf{the relative weight function} $\w:V^2\to\bR^+$ as
 \begin{equation}\label{def:haarweights}
 \w^y(x):=\w_{\Gamma}(x,y):=\frac{\abs{\Gamma_x y}}{\abs{\Gamma_y x}},
 \end{equation}
 where $\Gamma_v := \{\gamma\in\Gamma \mid \gamma v=v\}$ is the stabilizer of $v\in V$. We say that the group $\Gamma$ is \textbf{unimodular} if $\w^y(x)\equiv 1$, and the graph $G$ is \textbf{unimodular} if $\Aut(G)$ is. (For the standard definition of unimodularity via Haar measures and further details see Section~\ref{sec:nonunimod}.)
We often refer to the \textit{weight of a vertex}, fixing some arbitrary reference vertex $o$. Throughout this paper we omit writing the reference vertex $o$ in the superscript of $\w^o$, unless it is important to highlight. We also omit subscript $\Gamma$ when it is clear from the context.
We say that a set of vertices is \textbf{light} (resp.~\textbf{heavy}) if the sum of the weights is finite (resp.~infinite). Notice that such definitions do not depend on the choice of the reference vertex. When $\Gamma$ is unimodular, i.e.~$\w\equiv1$, any infinite cluster is automatically heavy. On the other hand, in the nonunimodular case, the light clusters could be either finite or infinite in size. Moreover, \cite[Theorem 4.1.6]{Haggstrom99} implies that  light-infinite clusters cannot coexist with heavy ones in Bernoulli percolation, yielding four percolation phases for any transitive graph (see Figure~\ref{Fig:phases}). Each of these phases can be nontrivial: see \cite{Haggstrom99,Timar06nonu,Pengfei18, Hutchcroft20} for details and examples.

\begin{rem}
    Since we are mostly concerned with the presence of weighs \textit{we do not restrict the scope of this paper to the weights induced by the full group of automorphisms}. Instead, similarly to \cite{BLPS99inv,Hutchcroft20}, we state our results for the weights induced by some closed subgroup $\Gamma\subseteq \Aut(G)$ that acts transitively on $G$. In particular, our analysis applies to all transitive graphs, but yields new results only for those whose automorphism group admits such a nonunimodular subgroup.
    
    For simplicity, we mostly restrict ourselves to the transitive case, but the definitions and results could be extended to the quasi-transitive setting without difficulty (see Lemma~\ref{lem:quasi-to-trans}).
    
\end{rem}

\begin{figure}[htb]
\begin{center}
	\begin{tikzpicture}[thick, scale=0.7]
	
	\draw [-] (-12,0) to (11,0);
 
        \draw (-12,0.3) to (-12,-0.3);
        \draw (-12,-0.3) node[below] {$0$};
        \draw (-7.5,0.3) to (-7.5,-0.3);
	\draw (-7.5,-0.3) node[below] {$p_c$};
        \draw (-0.5,0.3) to (-0.5,-0.3);
	\draw (-0.5,-0.3)node[below] {$p_h$};
        \draw (6.5,0.3) to (6.5,-0.3);
	\draw (6.5,-0.3) node[below] {$p_u$};
        \draw (11,0.3) to (11,-0.3);
        \draw (11,-0.3) node[below] {$1$};

        \draw (-10,1) node[above] {There are only};
        \draw (-10,0.3) node[above] {finite clusters};
        \draw (-4,1) node[above] {$\exists$ $\infty$-many infinite clusters};
        \draw (-4,0.3) node[above] {all of which are light};
        \draw (3,1) node[above] {$\exists$ $\infty$-many infinite clusters};
        \draw (3,0.3) node[above] {all of which are heavy };
        \draw (9,1) node[above] {$\exists\,!$ infinite cluster};
        \draw (9,0.3) node[above] {and it is heavy};
    \end{tikzpicture}
\end{center}
\caption{Depiction of four phases of Bernoulli$(p)$ percolation.}
\label{Fig:phases}
\end{figure}

\subsection{Weighted-amenability and its equivalents}
The present and the following subsection describe our main contributions to understanding and characterizing {\it weighted-amenability}, our central concept which will be defined soon. Figure~\ref{Fig:wamen-equiv} summarizes the various equivalents. 

One of the main conjectures in the field \cite[Conjecture 6]{BSbeyond} connects these phases to the geometry of the underlying graph. This conjecture says that \textit{a quasi-transitive graph $G$ is amenable if and only if $p_c = p_u$}, \ie\ there is no phase with infinitely many infinite clusters. Here by amenability of a graph, sometimes referred as \textbf{graph-amenability}, we mean that the Cheeger constant of the graph $\Phi_V(G)$ is zero. Recall the definition
\begin{equation}\label{eq:Cheeger}
    \Phi_V(G):=\inf_{K\subseteq V\atop \abs{K}<\infty}\frac{\abs{\partial_V K}}{\abs{K}},
\end{equation}
where $\abs{\cdot}$ denotes the size and $\partial_V$ denotes the external vertex boundary of a set.

In the case of unimodular graphs the amenability of a graph does coincide with the amenability of its automorphism group. This statement is known as the Soardi--Woess--Salvatori theorem, proved in \cite[Corollary 1]{Soardi-Woess} for transitive graphs and in \cite[Theorem 1]{Salvatori} for quasi-transitive ones. 

\begin{thm}[Soardi--Woess--Salvatori]\label{thm:Soardi--Woess--Salvatori}
A quasi-transitive graph is amenable if and only if its automorphism group is amenable and unimodular.
\end{thm}

Hence all nonunimodular graphs are nonamenable in the sense of \eqref{eq:Cheeger}. But some of them may have an amenable automorphism group. Moreover, such graphs can exhibit hyperfinite-like properties (meaning the existence of an exhaustion by increasing configurations of invariant percolation processes with only finite clusters), which is not possible in the unimodular nonamenable setting. In Benjamini, Lyons, Peres, and Schramm \cite[Section 3]{BLPS99inv} introduced weighted-amenability, or \textbf{$\w$-amenability} for short, and they extended the Soardi–Woess–Salvatori theorem by showing the equivalence of $\w$-amenability to group-amenability for graphs with a quasi-transitive automorphism group (see Subsection~\ref{sec:wamen}). The more general form of the definition allows for a graph and some transitive group of automorphisms on it:
\begin{defn}[Weighted-amenability, \cite{BLPS99inv}]\label{def:wamen}
Let $\Gamma$ be a closed subgroup of $\Aut(G)$ that acts transitively on $G$ and let $\w$ be the induced relative weight function as in \eqref{def:haarweights}. Then we say that $G$ is weighted-amenable or $\w$-amenable (with respect to the action of $\Gamma$) if 
\begin{equation}\label{eq:wCheeger}
    \Phi^{\w}_V(G):=\inf_{F\subseteq V\atop \abs{F}<\infty}\frac{\w^o(\partial_V F)}{\w^o(F)}=0,
\end{equation}
where $o$ is some reference vertex, $\abs{\cdot}$ denotes the size, and $\partial_V$ denotes the external vertex boundary, and for a set $A\subseteq V$ we define $\w^{o}(A):=\sum_{x\in A}\w^{o}(x)$. Otherwise we call $G$ weighted-nonamenable or $\w$-nonamenable (with respect to the action of $\Gamma$).
\end{defn}
Notice that weighted-amenability does not depend on the choice of the reference vertex $o\in V$. A similar notion of weighted isoperimetric constant was also introduced in the context of measured group theory by Kaimanovich \cite{kaimanovich1997amenability}.

A relative weight function on a graph defines a decomposition of the vertex sets into \textbf{levels}. We say that two vertices belong to the same level if they are of equal weight. It turns out that weighted-amenability is equivalent to {\bf level-amenability}, that is, the property that every finite union of levels induces an amenable subgraph. Level-amenability was first defined implicitly in \cite{Timar06nonu}, and it is at the center of \cite[Conjecture 8.5]{Hutchcroft20}. 
\begin{thm}\label{thm:levelamen-hf}
Let $\Gamma$ be a closed subgroup of $\Aut(G)$ that acts transitively on $G$ and $\w$ be the induced relative weight function as in \eqref{def:haarweights}. Then the following are equivalent:
\begin{enumerate}
    \item\label{thm:levelamen-hf:amen} $G$ is $\w$-amenable,
    \item\label{thm:levelamen-hf:lvl} $G$ is level-amenable,
    \item\label{thm:levelamen-hf:hf} $G$ is hyperfinite.
\end{enumerate}
\end{thm} 
The equivalence of \eqref{thm:levelamen-hf:amen} and \eqref{thm:levelamen-hf:hf} was essentially shown in \cite{BLPS99inv}.
More precise definitions and the proof are given in Section \ref{sec:levelamen}. Our proof of \eqref{thm:levelamen-hf:lvl}$\Rightarrow$\eqref{thm:levelamen-hf:hf} relies on percolation on levels of the graph, which uses that any transitive graph is \textbf{hyper-unimodular}, i.e.\ is an increasing union of invariant random unimodular subgraphs. Although hyper-unimodularity is simple to show it turns out to be very useful. In fact, a related approach was already used in the literature \cite{Timar06nonu,Pengfei18}, where authors considered so-called 1-partitions of the levels. An analogous idea was also recently independently used in measured group theory to resolve the problem of measure equivalence of Baumslag--Solitar groups \cite{MEBaumslagSolitar}.

Furthermore, we give a probabilistic characterization of weighted-amenability in  Theorem \ref{thm:BLPSnew}, which generalizes the next result of
Benjamini, Lyons, Peres, and Schramm.
\begin{thm}[{\cite[Theorem 5.3]{BLPS99inv}}]\label{thm:BLPSinv5.3}
    Let $\Gamma$ be a closed subgroup of $\Aut(G)$ that acts transitively on $G$. Then each of the following conditions implies the
next one:
    \begin{enumerate}
        \item G is amenable.
        \item\label{BLPS2} There is a $\Gamma$-invariant random spanning tree of $G$ with $\le 2$ ends a.s.
        \item\label{BLPS3} There is a $\Gamma$-invariant random connected subgraph $\go$ of $G$ with $p_c(\go) = 1$ a.s.
        \item $\Gamma$ is amenable.
    \end{enumerate}
    Moreover if $\Gamma$ is unimodular than all four conditions are equivalent.
\end{thm}

\noindent
Following the above theorem, they asked if the converse to the last implication holds and showed that it does in the unimodular case. The second author answered their question negatively by constructing a counterexample in the nonunimodular case \cite{Timar06nonu}. 
The original question of \cite{BLPS99inv} turns out to have a positive answer if we replace $p_c(\go)$ with $p_h(\go)$. Furthermore, the proper generalization of \cite[Theorem 5.3]{BLPS99inv} is obtained by considering weights everywhere: $\w$-amenability replacing amenability, $p_h$ replacing $p_c$, and nonvanishing ends replacing ends. An end $\xi$ is called \textbf{$\w$-vanishing} if $\w(x_n)$ converges to $0$ for any sequence of vertices $x_n$ that converges to $\xi$. Then the respective parts of \cite[Theorem 5.3]{BLPS99inv} (as in Theorem~\ref{thm:BLPSinv5.3}) turn into full equivalences:

\begin{thm}\label{thm:BLPSnew}
    Let $\Gamma$ be a closed subgroup of $\Aut(G)$ that acts transitively on $G$ and $\w$ be the induced relative weight function as in \eqref{def:haarweights}. Then the following conditions are equivalent:
    \begin{enumerate}
        \item\label{thm:BLPSnew1} $G$ is $\w$-amenable.
        \item\label{thm:BLPSnew2} There is a $\Gamma$-invariant random spanning tree of $G$ with $\le 2$ $\w$-nonvanishing ends a.s.
        \item\label{thm:BLPSnew3} There is a $\Gamma$-invariant random connected subgraph $\go$ with $p_h(\go)=1$ a.s.
        \item\label{thm:BLPSnew4} $\Gamma$ is amenable.
    \end{enumerate}
\end{thm}

\noindent
See Theorem \ref{thm:Forestequiv_new} for some further equivalents, in the spirit of \eqref{thm:BLPSnew2}. In the proofs of Theorems \ref{thm:BLPSnew} and \ref{thm:Forestequiv_new} we rely on the connection with measured group theory (via the cluster graphing construction) and recent results in that field \cite{AnushRobin,FmaxSF22}.

\begin{figure}[htb]

\begin{center}
	\begin{tikzpicture}[%
    auto,
    block/.style={
      rectangle,
      draw=black,
      thick,
      fill=black!10,
      text width=7em,
      align=center,
      rounded corners,
      minimum height=2em
    },
    block1/.style={
      rectangle,
      draw=black,
      thick,
      fill=black!10,
      text width=8em,
      align=center,
      rounded corners,
      minimum height=2em
    },
    block2/.style={
      rectangle,
      draw=black,
      thick,
      fill=black!10,
      text width=9em,
      align=center,
      rounded corners,
      minimum height=2em
    },
      block1red/.style={
      rectangle,
      draw=black,
      thick,
      fill=red!10,
      text width=8em,
      align=center,
      rounded corners,
      minimum height=2em
    },
    line/.style={
      draw,thick,
      -latex',
      shorten >=2pt
    },
    cloud/.style={
      draw=red,
      thick,
      ellipse,
      fill=red!20,
      minimum height=1em
    }
  ]
        
        \draw (-4,5) node[block2] (O) {$\forall \ga<1$, $\exists$ an invar. \\
         site perc.~$\mathbf{P}$ with $\mathbf{P}(o\in\go)>\ga$ and\\ no infinite cluster};
        \draw (-4,2) node[block] (A) {$\Aut(G)$ is amenable};
        \draw (3,2) node[block1red] (B) {$G$ is $\w$-amenable};
        \draw (3,-1) node[block2] (C) {$G$ is level-amenable};
        \draw (3,5) node[block2] (D) {$G$ is hyperfinite};
        \draw (-4,-1) node[block1] (E) {$p_h(G)=p_u(G)$};
        \draw (-3.7,-4.2) node[block2] (F) {$\exists$ increasing union \\of levels $\{G_n\}$ s.t. $$\lim_{n\to\infty}p_c(G_n)\ge p_u(G)$$};
        \draw (3,-4) node[block2] (G) {The spectral radius
        of the $\sqrt{\w}$-biased RW $$\rho^{(\w)}(G)=1$$};
        \draw (8.8,4.6) node[block2] (H) {$\exists$ an invariant random connected subgraph $\go$ with $p_h(\go)=1$};
        \draw (8.8,2) node[block2] (I) {$\exists$ invariant random \\spanning tree with $\le$ 2\\ {$\w$-nonvanishing} ends};
        \draw (8.8,-1) node[block2] (J) {All trees in any invariant random\\ forest have $\le$ 2\\ {$\w$-nonvanishing} ends};
        \draw (8.8,-4) node[block2] (K) {All trees in any invariant random forest have $p_h=1$};

        \draw [<->,dashed](-3,-0.5) to node [midway,sloped,above] {\scriptsize{{Conjecture~\ref{conj:ph_vs_pu}}}} (1.8,1.5);
        \draw[->] (-0.6,0.5) -- (-3,-0.5);
        
        \draw [<->] (-2.31,2) to node[above] {\scriptsize{{\cite[Theorem 3.9]{BLPS99inv}}}} (1.2,2);

        \draw [<->] (3,1.5) to node[right] {\scriptsize{Theorem~\ref{thm:levelamen-hf}}} (3,-0.5);
        \draw [<->] (3,4.5) to node [right] (TextNode1) {\scriptsize{{\cite[Theorem 5.1]{BLPS99inv}}}} (3,2.5);

        \draw [<-,dashed] (0.8,-1.2) to node [below] (TextNode){\scriptsize{{\cite[Conjecture 8.5]{Hutchcroft20}}}} (-2.1,-1.2);
        \draw (-0.65,-1.6) node[below] {\scriptsize{{Relaxation: Theorem~\ref{thm:wPSN}}}};
        
        \draw [->] (0.8,-0.8) to node [above]  {\scriptsize{{\cite[Corollary 5.8]{Timar06nonu}}}} (-2.1,-0.8);
        \draw [<->] (-4,-1.6) to node [right] {\scriptsize{{\cite[Corollary 4.8]{Pengfei18}}}} (-4,-3.1);
        \draw [<->] (-4,2.6) to node [right] {\scriptsize{{\cite[Theorem 5.1]{BLPS99inv}}}} (-4,3.9);
        
        \draw [<->] (4.5,1.5) to [bend left=30] node [above, sloped] (TextNode1) {\scriptsize{Theorem~\ref{thm:wKesten} (Kesten)}} (4.5,-2.8);
        
        \draw [<->](4.8,2) to node [above] {\scriptsize{{Theorem~\ref{thm:BLPSnew}}}} (6.8,2);
        \draw (5.8,2) -- (5.8,-1) node [midway,above, sloped] (TextNode1) {\scriptsize{Theorem \ref{thm:Forestequiv_new}}};
        \draw (5.8,-4) -- (5.8,-1);
        \draw [->] (5.8,-4) to (6.8,-4);
        \draw [->] (5.8,-1) to (6.8,-1);
        \draw (5.8,2.3) -- (5.8,3.3);
        \draw (5.8,3.7) -- (5.8,4.6);
        \draw [->] (5.8,4.6) to (6.8,4.6);
        
    \end{tikzpicture}
\end{center}
\caption{Depiction of the relations between various characterization of the weighted-amenability and relevant results.}
\label{Fig:wamen-equiv}
\end{figure}

\subsection{Weighted-amenability and Bernoulli percolation}
The fact that amenability of a graph implies $p_c=p_u$ was proven in \cite{pc=pu}; however the other direction of this conjecture is only known in particular cases that include transitive nonamenable graphs with suitably small spectral radius \cite{BSbeyond}, large Cheeger constant \cite{Schonmann01}, with large girth \cite{Nachmias12}, or for a suitably chosen Cayley graph of any nonamenable group \cite{PSN00}. Finally, this conjecture was resolved for graphs whose automorphism group contains a nonunimodular subgroup acting quasi-transitively \cite{Hutchcroft20}. In that work, Hutchcroft showed that for such graphs we have $p_c<p_h(\le p_u)$ and conjectured that level-amenability is equivalent to $p_h<p_u$ \cite[Conjecture 8.5]{Hutchcroft20}. While the forward direction was proven in \cite{Timar06nonu} in 2006, Hutchcroft suggested that the other direction is likely harder to prove than the ``$p_c<p_u$" conjecture of Benjamini and Schramm.

Our next result proves a relaxation of Hutchcroft's conjecture, as a nonunimodular analogue of the theorem of Pak and Smirnova-Nagnibeda on $p_c<p_u$
\cite{PSN00}.
\begin{thm}\label{thm:wPSN}
    Let $\Gamma$ be a closed subgroup of $\Aut(G)$ that acts transitively on $G$ and $\w$ be the induced relative weight function as in \eqref{def:haarweights}. If $G$ is $\w$-nonamenable then there exists a transitive graph $G'$ on the same vertex set that is quasi-isometric to $G$, $\Gamma$ acts transitively on $G'$, and such that $p_h(G',\Gamma)<p_u(G')$.
\end{thm}
Let us mention that to prove $p_c<p_h$, Hutchcroft analyzed yet another phase threshold $p_t$, corresponding to the so-called tiltability phase transition. From 
\cite{Timar06nonu} it follows that $p_t\le p_h$ and Hutchcroft showed that under the above assumptions, the inequality $p_c<p_t$ necessarily holds. For more on various phase transitions in the nonunimodular setting, we refer the reader to \cite{Hutchcroft20,HutchcroftPan24dim} and the references therein.
 
The next result shows the continuity of the heaviness phase transition for $\w$-nonamenable graphs for Bernoulli percolation. For unimodular transitive graphs (where $p_c=p_h$) such a result is due to Benjamini, Lyons, Peres, and Schramm \cite{BLPSpc}. In the nonunimodular setting, Tang proved a partial result. Namely, in \cite[Proposition 4.7]{Pengfei18} he utilized the slice approach of \cite{Timar06nonu} to show that there cannot be infinitely many heavy clusters at $p_h$, but our result does not rely on his approach.

\begin{thm}\label{thm:contph}
    Let $\Gamma$ be a closed subgroup of $\Aut(G)$ that acts transitively on $G$ and $\w$ be the induced relative weight function as in \eqref{def:haarweights}. Suppose $G$ is $\w$-nonamenable. Then 
    \begin{enumerate}
        \item\label{thm:contph1} there is no heavy cluster under Bernoulli$(p_h(G,\Gamma))$ bond percolation.
        \item\label{thm:contph2}  In particular, we have $p_h(G,\Gamma)<1$.
    \end{enumerate}
\end{thm}

That every nonamenable graphs has $p_c<1$ was proved in \cite{BSbeyond}. Part \ref{thm:contph2} of our theorem can be viewed as an analogue of that result in the weighted setting. See also Subsection \ref{subsec:ph}.

In \cite{Hutchcroft20}, Hutchcroft also asked if it is true that if $\Gamma'$ and $\Gamma\subseteq\Gamma'\subseteq\Aut(G)$ are quasi-transitive subgroups of automorphisms of a locally finite connected graph $G$, then $p_h(G,\Gamma)>p_h(G,\Gamma')$ if and only if $\w_\Gamma\neq \w_{\Gamma'}$, \cite[Question 8.6]{Hutchcroft20}. We found a negative answer to this question via a simple example, which shows that extra conditions are needed for the question. Consider an arbitrary transitive graph with a nonunimodular automorphism group, and add two extra leaves to each vertex, one yellow and one blue. Let $\Gamma'$ be the set of all automorphisms of the new graph (where the colors of the leaves are ignored), and $\Gamma$ be all those automorphisms that respect the colors of the leaves. Then $\w_\Gamma\not=\w_{\Gamma'}$, but  their ratio is between $1/2$ and 2 everywhere, so  $p_h(G,\Gamma)=p_h(G,\Gamma')$. One way to ask the question now is to assume transitivity instead of quasi-transitivity. Another possibility is to require that the ratio of the two weight functions is unbounded. In the transitive case this is equivalent to them being distinct, for the following reason. Suppose that $G$ is transitive and $\w_\Gamma\not=\w_{\Gamma'}$. We may assume that there is an edge $xy$ such that $\w_\Gamma(x)=\w_{\Gamma'}(x)$ and $\w_\Gamma(y)\not=\w_{\Gamma'}(y)$. Consider a $\gamma\in\Gamma$ that maps $x$ to $y$. Then the set $\{\w_\Gamma(\gamma^n(x))/\w_{\Gamma'}(\gamma^n(x)): n\geq 1\}$ in unbounded.

\subsection*{Organization of the paper}
We begin by reviewing the basic properties of nonunimodular graphs and introducing the corresponding weight function in Subsection~\ref{sec:nonunimod}. In Subsection~\ref{sec:wamen}, we recall the extension of the Soardi–Woess–Salvatori Theorem, along with several examples of nonunimodular graphs, while Subsection~\ref{sec:thresholds} adopts some inequalities linking isoperimetry and percolation into our weighted setup. Subsection~\ref{sec:ends} introduces weighted ends, Subsection~\ref{sec:mgt} summarizes some relevant results from measured group theory.
Section~\ref{sec:Kesten} defines the weighted simple random walk for nonunimodular graphs which makes $\w$ a stationary measure, and presents the extension of Kesten’s criterion for weighted-amenability through spectral radius. We also prove inequalities between the various weighted Cheeger constants that needed to be considered.
In these sections the proofs are mostly adaptations of existing techniques, our main contribution follows after these. In Section~\ref{sec:levelamen}, we introduce level-amenability and show its equivalence to weighted-amenability. In Section~\ref{sec:invarforest} we further characterize weighted-amenability in terms of invariant random spanning trees and forests (Theorem~\ref{thm:BLPSnew}).
Lastly, Section~\ref{sec:phasetrans} discusses results pertaining to percolation phase transitions, including the proofs of Theorems~\ref{thm:wPSN} and \ref{thm:contph}. We conclude the paper with open questions and further directions in Section~\ref{sec:next}.

%
\section{Preliminaries}\label{sec:preq}
%

\subsection{Nonunimodularity and relative weights}\label{sec:nonunimod}

Recall that the automorphism group $\Aut(G)$ of a connected locally finite graph $G := (V,E)$ is a locally compact group when equipped with the topology of pointwise convergence. It is a well-known fact, proven in \cite{Trofimov}, that the definition of unimodularity of $\Gamma\subseteq \Aut(G)$, presented in Subsection~\ref{sec:percintro}, is equivalent to its left Haar measure being right-invariant. We refer to \cite{BLPS99inv,LyonsBook} for an overview of unimodular automorphism groups and their significance for random subgraphs.

Let $\Gamma$ be a closed subgroup of $\Aut(G)$, $m$ be a left Haar measure on $\Gamma$, and $\w$ be the induced relative weight function as in \eqref{def:haarweights}. Then the the relative weight function $\w$ is equal to the \textbf{modular function}, \ie\ $\w^y(x)=m(\Gamma_x)/m(\Gamma_y)$. Furthermore, $\w$ is a $\Gamma$-invariant cocycle, meaning that the product of $\w^{y}(x)$ over the directed edges $(y,x)$ of any directed cycle is equal to $1$,
and  for any $x,y \in V$ and $\gamma\in\Gamma$ we have that
\begin{equation}\label{eq:invcoc}
    \w^y(x)
    =
    \frac{m(\Gamma_x)}{m(\Gamma_y)}
    =
    \frac{\abs{\Gamma_x y}}{\abs{\Gamma_y x}}
    =
    \frac{\abs{\Gamma_{\gamma x}\gamma y}}{\abs{\Gamma_{\gamma y}\gamma x}} 
    =
    \frac{m(\Gamma_{\gamma x})}{m(\Gamma_{\gamma y})}
    =
    \w^{\gamma y}(\gamma x).
\end{equation}
If $\Gamma$ acts quasi-transitively, by definition, there are only finitely many $\Gamma$-orbits, and hence if $\Gamma$ is also unimodular the set $\{\w^o(x) \mid x\in V\}$ is finite; the converse is also true: if $\{\w^o(x) \mid x\in V\}$ is finite, then $\Gamma$ is unimodular.

One of the main tools in the field is the Tilted Mass Transport Principle (TMTP), which we state for a transitive, possibly nonunimodular, graph. For various versions of this result see \cite{BLPS99inv,LyonsBook,Hutchcroft20}. We remark that in the literature the word ``tilted" did not appear until \cite{Hutchcroft20}, where a variant of the mass transport principle was introduced for quasi-transitive graph with an orbit representative chosen at random. However, nowadays it is used more broadly to mean the general mass transport principle tilted by the modular function.

\begin{thm}[Tilted Mass Transport Principle]\label{thm:TMTP}
    Let $\Gamma$ be a closed subgroup of $\Aut(G)$ that acts transitively on $G$ and $\w$ be the induced relative weight function as in \eqref{def:haarweights}. Then for any function $f:V\times V \to \bR^+$ that is invariant under the diagonal action of $\Gamma$ and for all $x,y\in V$ we have
    \[
    \sum_{z\in V} f(x,z)=\sum_{z\in V} f(z,y)\w^{y}(z).
    \]
\end{thm}

\begin{rem}[Absolute vs.~relative weights]
    It is possible to introduce \textit{absolute} weights on a graph by defining them to be equal to a left Haar measure of the stabilizer of a given vertex (it is finite because $\Gamma$ is a locally compact group). Since it is unique up to a multiplicative constant, choosing such a measure is equivalent to selecting a reference point $o\in V$. Thus one can write the MTP as in Section 3 of \cite{BLPS99inv}:
     \[
    \sum_{z\in V} f(x,z)\w^{o}(y)=\sum_{z\in V} f(z,y)\w^{o}(z).
    \]
   We chose to instead work with a \textit{relative} weight function, defined on pairs of vertices as in~\eqref{def:haarweights}. As we will see later, this allows us to connect it directly with the Radon--Nikodym cocycle of a corresponding countable equivalence relation (see Subsection~\ref{sec:mgt}).
\end{rem}

\subsection{Weighted-amenability, examples of nonunimodular graphs}\label{sec:wamen}
Benjamini, Lyons, Peres, and Schramm \cite[Theorem 3.9]{BLPS99inv} generalized the Soardi--Woess--Salvatori theorem in the combination of the following two results.
\begin{thm}[{\cite[Theorem 3.9]{BLPS99inv}}]\label{thm:BLPSw-amen}
    Let $\Gamma$ be a closed subgroup of $\Aut(G)$ that
acts transitively on $G$. Then $\Gamma$ is amenable if and only if $G$ is $\w$-amenable.
\end{thm}

By the theorem above, if $\Gamma$ is a closed subgroup of $\Aut(G)$ that
acts transitively on $G$, and there is a graph $G'$ on the vertex set of $G$ such that every element of $\Gamma$ is an automorphism of $G'$ then $G$ is $\w$-amenable if and only if $G'$ is.

\begin{lem}[{\cite[Lemma 3.10]{BLPS99inv}}]\label{lem:quasi-to-trans}
Let $G$ be a graph and $\Gamma$ be a closed quasi-transitive subgroup of $\Aut(G)$. For a vertex $o\in V$ let $r\in \bN$ be such that every vertex in $G$ is within distance $r$ of some vertex in $\Gamma o$. Form the graph $G'$ from the vertices $\Gamma o$ by joining two vertices by an edge if their distance in $G$ is at most $2r + 1$. Restriction of the elements of $\Gamma$ to $G'$ yields a subgroup $\Gamma'\subseteq \Aut(G)$. Then $G'$ is connected, $\Gamma'$ acts transitively on $G'$, and we have the following equivalences:         \begin{enumerate}
        \item  $G$ is $\w_\Gamma$-amenable if and only if $G'$ is $\w_{\Gamma'}$-amenable;
        \item $\Gamma$ is amenable if and only if $\Gamma'$ is amenable;
        \item $\Gamma$ is unimodular if and only if $\Gamma'$ is unimodular.      
    \end{enumerate} 
\end{lem}

\noindent
The amenability of the automorphism group of a graph was shown to be also equivalent to the existence of an invariant percolation with high marginal probabilities and no infinite clusters. We present a slightly rephrased version of this result here.
\begin{thm}[{\cite[Theorem 5.1]{BLPS99inv}}]\label{thm:BLPShf}
Let $\Gamma$ be a closed subgroup of $\Aut(G)$ that acts quasi-transitively on $G$ and $\w$ be the induced relative weight function as in \eqref{def:haarweights}. Then $G$ is $\w$-amenable if and only if for all $\ga<1$, there is a $\Gamma$-invariant
site percolation $\mathbf{P}_{\ga}$ on $G$ with $\mathbf{P}_{\ga}(x\in\go)>\ga$ for all $x$ and no infinite components. Analogous statement holds for bond percolation.
\end{thm}

In the context of nonunimodular transitive graphs one may wonder what happens if we replace finite (resp.~infinite) clusters with light (resp.~heavy) ones. In Definition~\ref{def:wamen} one can replace the condition $``\abs{F}<\infty"$ by $``\w^o(F)<\infty"$ (see Lemma~\ref{lem:finite->light}), and ``no infinite components" by ``no heavy components" in the conclusion of Theorem~\ref{thm:BLPShf}. Moreover, by a standard trick one can make the percolation processes $\mathbf{P}_\ga$ stochastically increasing in $\ga$, making the property in Theorem~\ref{thm:BLPShf} be the same as hyperfiniteness (see Definition \ref{def:hyperfinite}).

We now present several examples of nonunimodular graphs; for more examples see \cite{Timar06nonu,Pengfei18,Hutchcroft20} and references therein.

\begin{ex}
    Consider the Cayley graph of a countable finitely generated group, whose automorphism group contains a nonunimodular subgroup. One can modify this graph to ensure that the full automorphism group is nonunimodular by adding decorations (such as edges) to break undesired symmetries.  A specific case of a graph that is constructed in such a fashion is the grandparent graph presented in Example~\ref{ex:gp}.
\end{ex}

\begin{ex}[Nonunimodular trees $T_{r,s}$]\label{ex:nonuni-tree}
For any $r,s\in \bN$ such that $1\le r<s$. Consider a $r+s$ regular tree and an orientation of edges such that every vertex has exactly $r$ outgoing edges and $s$ incoming edges. Let $\Gamma$ be the subgroup of the automorphisms of the tree that preserves the orientation of the edges. It is easy to see that each vertex $x\in V$ has $r$ neighbors of weight $\w^x(\cdot)=s/r$ and $s$ neighbors of weight $\w^x(\cdot)=r/s$. A simple calculation that such a tree is weighted-amenable if and only if $r=1$.
\end{ex}

\begin{ex}[The grandparent graph $\mathrm{GP}(k)$]\label{ex:gp}
    The grandparent graph $\mathrm{GP}(k)$, $k\geq 2$, introduced by Trofimov \cite{Trofimov}, is constructed as follows. Consider $T_{1,k}$ from the previous example and connect every vertex to it's grandparent, \ie\ to the vertex of distance $2$ using outgoing edges. Notice that neighborhood of a vertex $x$ contains of a single parent ($\w^x(\cdot)=k$), a single grandparent ($\w^x(\cdot)=k^2$), $k$ children ($\w^x(\cdot)=1/k$), and $k^2$ grandchildren ($\w^x(\cdot)=1/k^2$). Weighted-amenability of $\mathrm{GP}(k)$ for any $k\geq 2$ follows from that of $T_{1,k}$.
\end{ex}

\begin{ex}[Diestel–-Leader  graph $\DL(k,\ell)$]\label{ex:DL}
The Diestel–-Leader graph $\DL(k, \ell)$, which was first introduced in \cite{DL01}, is constructed in the following way. Let $T_1:=T_{1,k}$ and $T_2:=T_{1,\ell}$ from Example~\ref{ex:nonuni-tree}. Align the levels of the tress in the opposite directions, \ie\ if $x\in T_1$ is the $i$-th level (with respect to an arbitrarily chosen root) then its children are in the $i+1$-th level; on the other hand for a vertex $y\in T_2$ is the $i$-th level then children are in the $i-1$-th level. Denote by $\level_{T_i}(x)$ the level of a vertex $x$ in $T_i$. Define a graph $\DL(k,\ell)$ on the vertex set
\[
\{(x, y) \in V(T_1) \times V (T_2): \level_{T_1}(x) = \level_{T_2}(y)\}. 
\]
with the edge set
\[
\left((x,y),(x',y')\right)\in E \Leftrightarrow (x,x')\in E(T_1) \textnormal{ and } (y,y')\in E(T_2).
\]
It is easy to check that $\DL(k,\ell)$ is nonunimodular whenever $k\neq\ell$. Moreover, each vertex $v$ has exactly $k$ neighbors such that $\w^x(\cdot)=\ell/k$ and $\ell$ neighbors such that $\w^x(\cdot)=k/\ell$. Similar calculations to those for the grandparent graph yield that the $\DL(k,\ell)$ graph is $\w$-amenable.
\end{ex}

\begin{ex}[Free product with nonunimodular graphs]\label{ex:GPxG}
Consider a graph that is the free product of a nonunimodular graph and a locally finite transitive graph $G$. It is clearly nonunimodular and it is not difficult to see that (as long as $G$ contains an edge) the weight of the boundary of any finite subgraph of such a product is at least that of the subgraph, yielding $\w$-nonamenability. Notice that taking the free product of two nonunimodular graphs where
the weights in the two are not the powers of the same constant, \eg\ $\mathrm{GP}(3)* \mathrm{GP}(4)$, results in a
nonunimodular graph where the
set of weights $\{3^i 4^j:\, i,j\in \bZ\}$ is dense in $\bR$.
\end{ex}

\begin{ex}[Cartesian product of $\mathrm{GP}(k)$ and $\bT_{d}$]\label{ex:CartprodTd}
Let $G$ be the Cartesian product of $\mathrm{GP}(k)$ and $\bT_d$. This graph is nonunimodular and if $d>2$, it is $\w$-nonamenable. Moreover, results from \cite{Hutchcroft20,Haggstrom99} together with simple calculation yield that, for $d$ large enough, all four phases from Figure~\ref{Fig:phases} are non-trivial for this graph.
\end{ex}

For all of the examples above the level sets are disconnected. Here by the \textbf{level} of $x\in V$ we mean a set of vertices $y$ such that $\w^x(y)=1$. Moreover, even though some of these graphs are one ended, deleting a level of vertices splits the graph into infinitely many components. The following example, to our knowledge, is the only construction of nonunimodular graphs with connected levels present in the literature. Coming up with other such graphs would be of interest.

\begin{ex}[Tower of ``self-similar" graphs]\label{ex:tower}
This construction was introduced in \cite{Timar06nonu}.
Let $\{G_i\}_{i\in \bZ}$ be copies of $\mathrm{DL}(2,2)$. Partition the vertices of $G_0$ into ``bags" of four vertices as follows. Let $T$ and $T'$ be the 3-regular trees used in the construction of $G_0$, as in Example~\ref{ex:DL}. If $x_1$ and $x_2$ are siblings in $T$
and $y_1$ and $y_2$ are siblings in $T'$, then let a bag be $\{(x_i,y_j)): i,j\in \{1, 2\}\}.$ Note that the bags are classes of imprimitivity for the automorphism group of $G_0$, \ie\ the partition to bags is preserved by any automorphism of $G_0$. Define the \textit{bag} graph $G_0'$ on the bags of $G_2$ by placing an edge between two bags if some element in one of them is adjacent (in $G_0$) to an element in the other. Notice that $G_0'$ is isomorphic to $\mathrm{DL}(2,2)$. Fix an arbitrary isomorphism $\phi:G_1 \to G_0'$ and connect a vertex $u\in G_1$ to a
vertex $v\in G_0'$ if the bag $\phi(u)$ contains $v$. Defining the edges between each $G_i$ and $G_{i+1}$ in the same way finishes the construction. This graph is indeed nonunimodular as for any $x\in G_i$ and $y\in G_{i+1}$ the weight of $y$ with respect to $x$ is $\w^x(y)=4$. It is easy to see that the resulting graph is $\w$-amenable. See \cite[Section 3.4]{Timar06nonu} for more details.

\end{ex}

\begin{ex}[Free product with the tower of ``self-similar" graphs]\label{ex:freeprodtower}
Consider a free product or direct product of the graph constructed in Example~\ref{ex:tower} with any nonamenable unimodular graph. Such a graph is $\w$-nonamenable. Then each level induces a connected {\it nonamenable} subgraph.
\end{ex}

\subsection{Percolation thresholds and weighted degrees}\label{sec:thresholds}

It is well-known that \textit{if an invariant percolation on a nonamenable unimodular transitive graphs only has finite clusters then the expected degree is bounded away from the degree in the graph} (see \cite[Theorem 8.16]{LyonsBook}), and the gap between the two can be expressed in terms of the isoperimetric constant. In the nonunimodular case analogous statements hold for the weighted degree, for which there are several natural candidates. In \cite[page 20]{BLPS99inv} the authors considered the following one (which they called $\Gamma$-degree):
for a vertex $x$ in a subgraph $F\subseteq G$ let the \textbf{weighted degree of $x$ in $F$} be
\begin{equation}\label{def:BLPS_D_F}
     \wh{D}^{\w}(x,F):=\sum_{z\in F\atop (x,z)\in F}\w^x(z).
\end{equation}

\begin{lem}
    Let $\Gamma$ be a closed subgroup of $\Aut(G)$ that acts transitively on $G$ and $\w$ be the induced relative weight function as in \eqref{def:haarweights}. Suppose $d$ is the degree in $G$. Then
    \begin{equation}\label{eq:wD=d}
    \wh{D}^{\w}(o,G)=\sum_{z\in V\atop o\sim z}\w^o(z)=d.
    \end{equation}
\end{lem}
\begin{proof}
    Let $\delta\in \bR^+$ such that there is a vertex $z\in V$ that is a neighbor of $o$ and $\w^o(z)=\delta$. Define
    \begin{align*}
        N_{\delta}=\abs{\{z\sim o \mid \w^o(z)=\delta\}}
    \end{align*}
    and notice that we have $\delta N_\delta=N_{1/\delta}$, which follows from the TMTP where each $x$ send unit mass to every element of $N_\delta$.
    Then
    \[
    \wh{D}^{\w}(o,G)=\sum_{\delta\in\bR^+}\delta N_\delta=1\cdot N_1+\sum_{\delta>1}\delta N_\delta+\frac{1}{\delta} N_{1/\delta}=N_1+\sum_{\delta>1}N_{1/\delta}+N_{\delta}=d.\tag*{\qedhere}
    \]
\end{proof}

The choice of weighted degree as in \eqref{def:BLPS_D_F} leads to the following version of the weighted edge-isoperimetric constant:
\begin{equation}\label{def:BLPS_edgeiso}
    \iota_E^{\,\w}(G)=\inf_{F\subseteq V\atop \abs{F}<\infty}\frac{1}{\w^o(F)}\sum_{x\in F\atop (x,z)\in \partial_E F}\w^o(x),
\end{equation}
which is not dependent on the choice of $o$.

\begin{lem}
    Let $\Gamma$ be a closed subgroup of $\Aut(G)$ that acts transitively on $G$ and $\w$ be the induced relative weight function as in \eqref{def:haarweights}. Suppose $d$ is the degree in $G$. Then
    \begin{equation}
        \frac{1}{d}\Phi_V^{\w}(G)\le\iota_E^{\,\w}(G)\le d \Phi_V^{\w}(G).
    \end{equation}
\end{lem}
\begin{proof}
    Notice that from \eqref{def:haarweights} for any pair of adjacent vertices $x$ and $z$ we have that $d^{-1}\le w^{x}(z)\le d$. Thus for any finite set $F$ we have that
    \begin{align*}
        \sum_{x\in F\atop (x,z)\in \partial_E F}\w^o(x)\ge\sum_{x\in F\atop (x,z)\in \partial_E F}\w^o(x)\frac{\w^{x}(z)}{d}\ge\frac{1}{d}\sum_{z\in \partial_V F}\w^{o}(z).
    \end{align*}
    Similarly,
    \begin{align*}
        \sum_{x\in F\atop (x,z)\in \partial_E F}\w^o(x)\le 
        \sum_{z\in \partial_V F}\w^o(x)\w^{x}(z) \w^{z}(x) \cdot \abs{\{x\in F\mid x\sim z\}}
        \le d\sum_{z\in \partial_V F}\w^{o}(z),
    \end{align*}
    because $\w^{z}(x) \cdot \abs{\{x\in F\mid x\sim z\}}\leq \wh{D}^{\w}(x,G)=d$.
\end{proof}
We now present nonunimodular versions of \cite[Theorem 8.16]{LyonsBook} that were essentially proved in \cite[Theorem 4.4]{BLPS99inv}. The only novelty in the result below that it is stated for light (resp.~heavy) clusters instead of finite (resp.~infinite) ones, but the proofs remain unchanged.

\begin{thm}\label{thm:Ewh{D}inv}
     Let $\Gamma$ be a closed subgroup of $\Aut(G)$ that acts transitively on $G$ and $\w$ be the induced relative weight function as in \eqref{def:haarweights}. Suppose that $\mathbf{P}$ is a $\Gamma$-invariant bond percolation in $G$. If $\mathbf{P}$-a.s.\ all clusters are light, then
     \begin{equation}\label{eq:EwhDinv}
     \mathbf{E} \wh{D}^{\w}(x,\go) \leq d-\iota_E^{\,\w}(G),  
     \end{equation}
     where $d$ is the degree in $G$.
\end{thm}
\begin{proof}
    For a subgraph $F\subseteq G$ define
    \begin{equation*}
        \wh{A}^{\,\w}(F):=\frac{1}{\w^o(F)}\sum_{x,y \in F \atop (x,y)\in F} \deg_F(x)\cdot \w^o(x){=\frac{1}{\w^o(F)}\sum_{(x,y) \in E(F)} \w^o(x)+\w^o(y)}
        \end{equation*}
        and
    \begin{equation*}
        \wh{A}^{\,\w}(G):=\sup_{F\subseteq G\atop \abs{F}<\infty} \wh{A}^{\,\w}(F).
    \end{equation*}
    Let $C(x)$ be the cluster containing vertex $x\in V$ and let $\ga_x:= \w^o(C(x))\in (0,\infty)$. Consider the following mass transport scheme: each vertex $x$ sends to every vertex $y$ in the same cluster (including itself) the mass of $\ga_x^{-1} \w^o(y) \wh{D}^{\w}(x,\go)$. In other words, $x$ redistributes a mass equal to its weighted-degree among all vertices in the same cluster proportional to their contribution to the total weight of the cluster. 

    Since $\sum_{y\in C(x)}\ga_x^{-1}\w^o(y)=1$ we have that expected sent-out mass from $x$ is equal to
        \[
        \mathbf{E} \wh{D}^{\w}(x,\go)=\mathbf{E}\sum_{y\in C(x)}\frac{\w^o(y)}{\ga_x}\wh{D}^{\w}(x,\go).
        \]
    On the other hand, the expected received mass at $x$ is given by 
        \begin{align*}\mathbf{E}\sum_{y\in C(x)}\frac{\w^o(x)}{\ga_x}\wh{D}^{\w}(y,\go)\w^x(y)
        &=\mathbf{E}\,\ga_x^{-1}\sum_{y\in C(x)}\w^o(y)\wh{D}^{\w}(y,\go)\\
        &=\mathbf{E}\,\ga_x^{-1}\sum_{y\in C(x)}\w^o(y)\sum_{z\in C(x)\atop (y,z)\in E(\go)}\w^y(z)\\
        &=\mathbf{E}\,\ga_x^{-1}\sum_{y\in C(x)}\sum_{z\in C(x)\atop (y,z)\in E(\go)}\w^o(z)=\mathbf{E} \wh{A}^{\,\w}(C(x)).
    \end{align*}
    From TMPT and the fact that $\iota_E^{\,\w}(G)+\wh{A}^{\w}(G)=d$ we conclude that
    \[
    \mathbf{E} \wh{D}^{\w}(x,\go)=\mathbf{E} \wh{A}^{\,\w}(C(x))\le \wh{A}^{\,\w}(G)=d-\iota_E^{\,\w}(G).\tag*{\qedhere}
    \]
\end{proof}
The previous theorem could be used to derive explicit upper bounds for $p_h$, which also implies part (ii) of Theorem~\ref{thm:contph}.

\begin{cor}\label{thm:ph_threshold}
     Let $\Gamma$ be a closed subgroup of $\Aut(G)$ that acts transitively on $G$ and $\w$ be the induced relative weight function as in \eqref{def:haarweights}. Suppose $\mathbf{P}$ is a $\Gamma$-invariant bond percolation on $G$. If 
     \[
     \mathbf{E} \wh{D}^{\w}(x,\go)>d-\iota_E^{\,\w}(G)
     \] then with positive probability there is a heavy cluster in $\go$. Moreover, with positive probability there is a cluster $C\subseteq \go$ such that $p_h(C,\gC)<1$.  
\end{cor}
\begin{proof}
Theorem~\ref{thm:Ewh{D}inv} yields that a configuration $\go$ contains a heavy cluster with positive probability. Considering such an event, notice that the expected degree in $\go$ under Bernoulli$(p)$ percolation on it is continuous in $p$. Thus, for $p$ large enough, the inequality \eqref{eq:EwhDinv} is violated, implying that there is a cluster $C\subseteq\go$ with $p_h(C,\gC)<1$.
\end{proof}

From Theorem~\ref{thm:BLPShf} it follows that $G$ is $\w$-nonamenable if and only if there is $\ga_c\in(0,1)$ such that every $\Gamma$-invariant site percolation on $G$, whose marginal probability exceeds $\ga_c$, has an infinite cluster with positive probability. 
For a quantitative bound, Peres derived in \cite[Theorem 2.3]{Peres00} that for a quasi-transitive $\w$-nonamenable graph $G$ 
\[
\ga_c\le \frac{\max_{x\in V}\deg_G(x)}{\Phi^{\w}_V(G)+\max_{x\in V}\deg_G(x)}.
\] The following statement is a slight generalization of this result as we will find bound for \textit{heavy} clusters, although the argument is essentially the same. We focus on the transitive case, a quasi-transitive version follows similarly.

\begin{thm}\label{thm:site-threshold}
    Let $\Gamma$ be a closed subgroup of $\Aut(G)$ that acts transitively on $G$ and $\w$ be the induced relative weight function as in \eqref{def:haarweights}. Suppose $G$ is of degree $d$ and is $\w$-nonamenable, and  $\mathbf{P}$ is a $\Gamma$-invariant site percolation on $G$, such that
    \begin{equation}\label{eq:site-threshold}
    \mathbf{P}(o\in\go)\ge\frac{d}{\Phi^{\w}_V(G)+d}.
    \end{equation}  
Then $\go$ contains a heavy cluster with positive probability. 
\end{thm}

\begin{proof}
    Consider the following mass transport scheme
    \[
    F(x,y,\go)=\frac{\w^x(y)}{\w^x(\partial_V C(x))}\ind_{\w^x(\partial_V C(x))<\infty}\ind_{x\in \go}\ind_{y\in\partial_V C(x)}.
    \]
    In other words, each vertex $x$ from a light cluster distributes unit mass to the vertices in the external vertex boundary of its cluster $C(x)$ proportional to their weights. 
    Notice that for all $x\in V$ 
    \begin{align}\label{eq:site-mass-out}
        \mathbf{E}\sum_{z\in V}F(x,z,\go)&=\mathbf{E}\sum_{z\in V}\frac{\w^x(z)}{\w^x(\partial_V C(x))}\ind_{\w^x(\partial_V C(x))<\infty}\ind_{x\in \go}\ind_{z\in\partial_V C(x)}\notag\\
        &=\mathbf{P}\left(\w^x(\partial_V C(x))\in(0,\infty)\right)
    \end{align}
    and
    \begin{align}\label{eq:site-mass-in}
         \mathbf{E}\sum_{z\in V}F(z,x,\go)\w^{x}(z)&=\mathbf{E}\sum_{z\in V}\frac{\w^z(x)}{\w^z(\partial_V C(z))}\w^x(z)\ind_{\w^x(\partial_V C(z))<\infty}\ind_{z\in \go}\ind_{x\in\partial_V C(z)}\notag\\
         &=\mathbf{E}\sum_{z\in V}\frac{\w^x(z)}{\w^x(\partial_V C(z))}\ind_{\w^x(\partial_V C(z))<\infty}\ind_{z\in \go}\ind_{x\in\partial_V C(z)}\notag\\
         &\le \frac{d}{\Phi^{\w}_V(G)}\pr(x\notin \go),
    \end{align}
    where the last inequality follows from the fact that $x$ must be closed to be in the boundary of $C(z)$ and it can be adjacent to at most $d$ many open clusters in $\go$.

    Since \eqref{eq:site-mass-out} counts the expected outgoing mass from $x$ and \eqref{eq:site-mass-in} is an upper bounds on the expected incoming mass to $x$, by the TMTP we have that
    \[
    \mathbf{P}\left(\w^x(\partial_V C(x))\in(0,\infty)\right)\le\frac{d}{\Phi^{\w}_V(G)}\pr(x\notin \go).
    \]
    Thus
    \begin{align*}
    \mathbf{P}\left(\w^x(\partial_V C(x))<\infty\right)&\le\pr(x\notin \go)+ \frac{d}{\Phi^{\w}_V(G)}\pr(x\notin \go)\\
    &=\frac{d+\Phi^{\w}_V(G)}{\Phi^{\w}_V(G)}(1-\pr(x\in \go))<1,
    \end{align*}
    where the last inequality follows from \eqref{eq:site-threshold}, completing the proof.
\end{proof}

\subsection{Ends of a graph}\label{sec:ends}
We follow a traditional combinatorial approach to ends \cite[p.~242]{LyonsBook}, for an equivalent point of view via end compactification see \cite[Section 2.B]{FmaxSF22}.

We say that a set of vertices $A$ is \textbf{end-convergent} if for any finite subgraph $K\subset G$ all but finitely many elements of $A$ are contained in the same connected component of $G\setminus K$.  Two end-convergent sets $A$ and $B$ are said to be equivalent if $A\cup B$ is end convergent. Notice that for locally finite graphs any infinite set of vertices contains an end-convergent subsequence. An \textbf{end} of $G$ is an equivalence class of end-convergent sets. We say that a set of vertices \textbf{converges to an end} $\xi$ if it belongs to the equivalence class $\xi$. Finally, we say that an end $\xi$ contains a light (resp.~heavy) \textbf{geodesic path} if there is a path in $G$ such that its vertex set converges to $\xi$ and such that no finite subpath can be replaced by a shorter path, and $\sum_i \w^o(x_i)<\infty$ (resp.~$\sum_i \w^o(x_i)=\infty$).

We now define $\w$-vanishing and heavy ends. The former was originally introduced in \cite{FmaxSF22, AnushRobin}. Given a relative weight function  $\w^o : V \to \bR^+$ on $G$ with respect to some reference point $o$, we call a set $A \subseteq V$ $\w$-\textbf{vanishing} if $$\limsup_{x \in A} \w(x)=0;$$ otherwise $A$ is  $\w$-\textbf{nonvanishing}. 
Finally, we call an end $\xi$ of $G$ $\w$-\textbf{vanishing} if any set $A\subset G$ that converges to it is $\w$-vanishing; otherwise $\xi$ is $\w$-\textbf{nonvanishing}. In particular, if $\xi$ is $\w$-\textbf{nonvanishing} then there is a sequence of vertices $\{x_n\}_{n\in\bN}$ that converges to $\xi$ and whose weights are bounded away from zero. 

Finally, we call an $\xi$ a \textbf{heavy end} if for any finite edge-cut the connected component that contains $\xi$ is heavy. Notice that these definitions do not depend on the choice of the reference point $o$ and for a unimodular $G$ $\w$-nonvanishing and heavy ends reduce to the classical ends.

The next standard lemma follows directly from the TMTP: let each vertex distribute mass $1$ among vertices of maximal weight in its cluster, if there are finitely many of those.
\begin{lem}\label{lem:maxweight}
    Let $\Gamma$ be a closed subgroup of $\Aut(G)$ that acts transitively on $G$ and let $\w$ be the induced relative weight function as in \eqref{def:haarweights}. Suppose $\mathbf{P}$ is a $\gC$-invariant percolation on $G$. Then any cluster that has finitely many vertices of maximal weight is light a.s.
\end{lem}

\begin{cor}\label{lem:noends}
    Let $\Gamma$ be a closed subgroup of $\Aut(G)$ that acts transitively on $G$ and let $\w$ be the induced relative weight function as in \eqref{def:haarweights}. Any heavy cluster of a $\gC$-invariant percolation on $G$ contains a $\w$-nonvanishing end.
\end{cor}
\begin{proof}
    By Lemma~\ref{lem:maxweight} a heavy cluster has to contain a sequence of vertices with nondecreasing weights. Since $G$ is locally finite such a sequence contains an end-convergent subsequence.
\end{proof}

\subsection{Measured group theoretic results}\label{sec:mgt}
As we mentioned above, some of our motivation came from recent progress in understanding the amenability of mcp (measure class preserving) Borel graphs in measured group theory. So, in this section, we briefly present relevant material. 

Measured group theory is concerned with studying groups from the point of view of their measurable actions on a standard probability space $(X,\mu)$. It is intertwined with the studies of locally countable Borel graphs and countable Borel equivalence relations (CBERs) as they naturally arise as Schreier graphs and orbit equivalence relations of such actions respectively. The converse also holds, by the Feldman--Moore theorem \cite{Feldman-Moore}. The notion of amenability was extended from groups to CBERs and in the measurable setting is called \textbf{$\mu$-amenability} \cite{Zimmer78}. Here for simplicity we define $\mu$-hyperfiniteness of a CBER, which is equivalent to $\mu$-amenability of the CBER by
the Connes--Feldman--Weiss theorem \cite{CFW}. 
A CBER $\cR$ is called \textbf{$\mu$-hyperfinite} if it is a countable increasing union of Borel equivalence relations with finite classes $\mu$-a.e.  

We call a CBER $\cR$ on a standard probability space $(X,\mu)$ probability measure preserving (\textbf{pmp}) (resp.\ measure class preserving or \textbf{mcp}) if for any Borel automorphism $\gamma$ on $X$ that maps every point to an $\cR$-equivalent point $\gamma$ preserves $\mu$ (resp.\ $\mu$-null sets). It follows from the argument of Kechris and Woodin, see \cite[Proposition 2.1]{Miller:thesis} or \cite[2.2]{TsZomback}, that any CBER $\cR$ on a probability space $(X,\mu)$
is mcp on a $\mu$-conull set. The lack of invariance of mcp CBERs is quantified similarly to Theorem~\ref{thm:TMTP}. This is made precise by the \textbf{Radon--Nikodym cocycle} $(x,y) \mapsto \w_\nu^y(x) : \cR \to \bR^+$ of the orbit equivalence relation $\cR$ with respect to the underlying probability measure $\mu$, as defined in \cite[Section 8]{KMtopics}. Here the  cocycle $\w_\nu^y(x)$ can be also interpreted as a relative weight function on $X$, hence we use the same notation as we did for graphs whenever the context is clear. The Radon--Nikodym cocycle ``corrects'' the failure of invariance of the measure $\mu$, enabling the (tilted) mass transport principle (see \cite[Section 8]{KMtopics} and \cite[Page 18]{Gaboriau05}): for each $f : \cR \to [0, \infty]$ we have that

\begin{equation}\label{eq:BorelMTP}
    \int \sum_{z \in [x]_{\cR}} f(x,z) d\mu(x) = \int \sum_{z \in [y]_{\cR}} f(z,y) \w_\nu^y(z) d\mu(y),
\end{equation}

where $[x]_{\cR}$ denotes the equivalence class of $x$ in $\cR$. Like for graphs, where $\w_\nu\equiv1$ case of Theorem~\ref{thm:TMTP} characterizes unimodular graphs, here the Radon--Nikodym cocycle $\w_\nu^y(x)\equiv 1$ is exactly when $\cR$ is pmp. Hence, just as unimodular transitive graphs turned out to be more approachable, much of the measured group theory has been developed under the assumption of pmp. A big part of this is due to the fact that the theory of cost is developed only for pmp CBERs \cite{Gaboriau:mercuriale}. The \textbf{cost} of a pmp CBER $\cR$ is defined as the infimum of the cost (i.e.\ half of expected degree) of its \textbf{graphings} (Borel graphs $G$ on $X$ whose connectedness relation is equal to $\cR$ a.e.). One can interpret cost is an analog for CBERs of the free rank for groups. Unfortunately, most of the theory of pmp actions and CBERs fails in the mcp setting, for instance, nonamenability of the group does not imply $\mu$-nonamenability of the orbit equivalence relations of its free mcp actions.
See \cite{FmaxSF22} and the references therein for further discussion. 

It is a major open question if there is a mcp analog of cost (see discussion and related results in \cite{poulin-anticost}). Before presenting some of the recent results in the literature we would like to highlight that both theory of cost and classical techniques from percolation theory rely on the value of the expected degree. Thus, they do not depend on the behavior of the Radon–-Nikodym cocycle or the modular function, respectively. In the present paper, as we already seen in Theorem~\ref{thm:Ewh{D}inv} some of the results utilize expected \textit{weighted} degrees (the key in defining cost) to yield desired implications on the percolation side of the interplay. Although it still does not generalize all of the applications of the expected degree from the unimodular setting, we believe it is a step in the right direction.

A Borel graph $\cG$ is called pmp/mcp/$\mu$-amenable/$\mu$-hyperfinite if its connectedness equivalence relation is such. Before the development of the theory of cost, the main approach to studying $\mu$-amenability of pmp graphs was through the number of ends. For instance, Adams \cite{Adams:trees_amenability} showed that an acyclic ergodic pmp graph is $\mu$-amenable exactly when it has $\le 2$ ends a.e. Hence, in the absence of the notion of cost in the mcp setting, it is natural to consider the geometry of the mcp graph, this time together with the behavior of the Radon–-Nikodym cocycle. In \cite{AnushRobin}, Tserunyan and Tucker-Drob generalized the result of Adams to the mcp setting as follows: 
\begin{thm}[{\cite{AnushRobin}}]\label{thm: Anush-Robin}
Let $\cG$ be an acyclic mcp graph on a standard probability space $(X,\mu)$ and $\w_\mu$ be the Radon--Nikodym cocycle of the connectedness relation $\cR_{\cG}$ with respect to $\mu$. Then the following are equivalent:
\begin{enumerate}
    \item $\cG$ is $\mu$-amenable,
    \item  a.e.\ $\cG$-component has $\le 2$ $\w_\mu$-nonvanishing ends,
    \item  a.e.\ $\cG$-component has $\le 2$ ends with heavy geodesic paths.
\end{enumerate}
Moreover, if a.e.\ $\cG$-component has $\le 2$ $\w_\mu$-nonvanishing ends then all other ends have only light geodesic paths.

Here $\w_\mu$-nonvanishing end and light geodesic paths mean the same as in Subsection~\ref{sec:ends}, but with respect to the Radon--Nikodym cocycle $\w_\mu$.
\end{thm}

The main result of \cite{FmaxSF22} is an explicit construction of a $\mu$-nonamenable subforest of $\cG$, which generalizes a theorem of  Gaboriau--Ghys (\cite{Ghys:Stallings} and \cite[IV.24]{Gaboriau:cout}) to the mcp setting for locally finite graphs.
\begin{thm}[{\cite{FmaxSF22}}]\label{thm:FmaxSF}
    Let $\cG$ be a locally finite mcp graph on a standard probability space $(X,\mu)$ and $\w_\mu$ be the Radon--Nikodym cocycle of the connectedness relation $\cR_{\cG}$ with respect to $\mu$. 
    If a.e.\ $\cG$-component has $\ge 3$ $\w_\mu$-nonvanishing ends, then there is a Borel subforest $\cF \subseteq \cG$ such that for a.e.\ $\cF$-connected component, the space of $\w_\mu$-nonvanishing ends of that component is nonempty and perfect (no isolated points). 
    In particular, $\cG$ is $\mu$-nonamenable. 
    Moreover, $\cF$ can be made ergodic if $\cG$ is.
\end{thm}

The connection between measured group theory and percolation theory is largely facilitated by the general \textbf{cluster graphing} construction \cite[Section 2.2 and 2.3]{Gaboriau05}, see \cite[Section 5B]{FmaxSF22} for a presentation tailored to the nonunimodular/mcp settings. Here we briefly summarize the general construction.  Observation \ref{obs:cluster} summarizes all the properties of the construction we will use, hence the reader may continue from there. 

Let $G = (V,E)$ be a locally finite connected rooted graph on a vertex set $V$ with root $o\in V$, $\Gamma \subseteq \Aut(G)$ be a closed subgroup, and $\mathbf{P}$ be a $\Gamma$-invariant bond percolation on $G$. In the case when $\Gamma$ is a countable finitely generated group and $G$ is its Cayley graph with respect to a symmetric set of generators $\Sigma$, the cluster graphing can be defined in a straightforward way (\cite{Gaboriau-Lyons,PetePGG}) as follows. Consider the shift action of $\Gamma$ on $2^E$ and the corresponding Schreier graph $\cG$, namely place an edge between $\go,\go'\in 2^E$ whenever there exists $\gamma\in \Sigma$ such that $\gamma \go=\go'$. Then the cluster graphing $\cG^\mathrm{cl}$ is defined as a subgraphing of $\cG$ where the edge $(\go,\go')\in \cG$ is retained if and only if $(o,\gamma o)\in \go$. Notice that this construction is $\Gamma$-invariant, \ie\ in this case $\gamma^{-1}\go'=\go$ and $(o,\gamma o)\in \go$ if and only if $(o,\gamma^{-1} o)\in \go'$.

For a general closed group $\Gamma \subseteq \Aut(G)$ the definition of a cluster graphing requires a few technical steps. Formally one considers an ergodic free pmp action $\Gamma$ on a standard probability space $(X,\mu)$ that factors onto $(2^E,\mathbf{P})$ (\ie\ admits a $\Gamma$-equivariant projection $\pi : X \to 2^E$ such that $\pi_*(\mu)=\mathbf{P}$) and its extension to the diagonal action of $\Gamma$ on $X\times V$ (we think of it as assigning root vertices to elements of $X$). The cluster graphing is then defined on a standard probability space $(Y,\nu)$, where $Y=(X\times V)/\Gamma$. Fixing an $o\in V$, for each element of $Y$ we can take the representative of the form $(x,o)$, and define probability measure $\nu$ on $Y$ as $\nu(A):=\mu(\{x\in X, [(x,o)]_\Gamma \in A
\})$. Now, the action of $\Gamma$ on $Y$ may not preserve $\nu$ (as is the case when $\Gamma$ in nonunimodular), but it is always mcp on a $\nu$-conull set (as mentioned above e.g.\ by \cite[Proposition 2.1]{Miller:thesis}).
Similarly to the above, we first define a graphing $\cG$ whose connected components are isomorphic to $G$ (also known as the graphing of the rerooting relation): by setting two elements $\wt{x}:=[(x,u)]_\Gamma$ and $\wt{y}:=[(y,v)]_\Gamma$ from $Y$ to be adjacent in $\cG$ if and only if there exists $\gamma\in \Gamma$ so that $x=\gamma y$ and $(u,\gamma v)\in E$. That is, one can choose different representatives so that their first coordinates are the same, and their second coordinates are connected by an edge in $G$.
Note that since $\pi$ is $\Gamma$-equivariant it is well defined on $Y$ (for any $\wt{x}=[(x,u)]_\Gamma$ we have $\pi(\wt{x})=\pi(x)$). We now define the cluster graphing $\cG^\mathrm{cl}\subseteq \cG$ by retaining the edges $([(x,u)]_\Gamma,[(x,v)]_\Gamma)$ of $\cG$ such that $(u,v)\in\pi(x)$. Thus, the $\cG^{\mathrm{cl}}$-component of a point $\wt{x}=[(x,o)]_\Gamma \in Y$ is isomorphic to the cluster of the root $o$ in the percolation configuration $\pi(\wt{x})=\pi(x)$. 
Moreover, the Radon--Nikodym cocycle of $\cG^{\mathrm{cl}}$  with respect to $\nu$ is given by the relative weight function $\w_\Gamma$ on $G$ induced by $\Gamma$ (as defined in \eqref{def:haarweights}), see \cite[Lemma 5.7]{FmaxSF22} for the proof.

The following observation summarizes its key properties.
\begin{obs}\label{obs:cluster} Amenability of a closed group $\Gamma\subseteq\Aut(G)$ and the properties of a $\Gamma$-invariant percolation $\mathbf{P}$ on $G$ are connected to the properties of the cluster graphing $\cG^{\mathrm{cl}}$ via the following results.
\begin{enumerate}[label=(\alph*)]
    \item\label{obs:cluster-iso} The $\cG^{\mathrm{cl}}$-component of $\wt{x}=[(x,o)]_\Gamma\in Y$ is isomorphic to the cluster of $o$ in $\pi(x)$.
    \item\label{obs:cluster-pmp} The Radon--Nikodym cocycle $\w_\nu$ of $\cG^{\mathrm{cl}}$ with respect to $\nu$ is induced by $\w_\Gamma$ (as in \eqref{def:haarweights})
    \begin{equation}\label{eq:RNweights}
        \w_\nu([x,v]_\Gamma,[x,u]_\Gamma):=\w_\nu^{[x,u]_\Gamma} ([x,v]_\Gamma):=\w_\Gamma^{u}(v),
    \end{equation}
    \item\label{obs:cluster-amen} If $\cG^{\mathrm{cl}}$ is nowhere amenable then $\Gamma$ is nonamenable. This follows from \cite[Proposition 6]{nonugroups} and \cite[Page 49]{JKLcber}.
\end{enumerate}    
\end{obs}

%
\section{$\sqrt{\w}$-biasing}\label{sec:Kesten}
%
Classically, one of the ways to characterize the amenability of groups and graphs is via the spectral radius of a simple random walk on its Cayley graph. 

\begin{defn}[Spectral radius]
    If $X_n$ is an irreducible Markov chain on $V(G)$ and $p_n(x,y):=\pr_x(X_n=y)$, then we define the spectral radius of $G$ as 
    \[
    \rho(G):=\lim\sup_{n\to \infty} p_{2n}(x,x)^{\frac{1}{2n}},
    \]
where $(X_n)$ is simple random walk with $X_0=x$.
    Notice that $\rho(G)$ is independent from the choice of the vertex $x\in V$.
\end{defn}

Recall Kesten's characterization of amenability for groups via the spectral radius.
\begin{thm}[Kesten \cite{Kesten}]
    Let $\Gamma$ be a finitely generated group, $S$ be a finite symmetric set of generators, and denote the corresponding Cayley graph by $\mathrm{Cay}(\Gamma,S)$. Then $\Gamma$ is amenable if and only if the spectral radius $\rho(\mathrm{Cay}(\Gamma,S))$ of the simple random walk is 1.
\end{thm}

This classification extends beyond Cayley graphs, to infinite networks (that is, graphs with positive weights, also known as conductance, on edges, and with Cheeger constant defined accordingly) and the associated network random walk, where the random walk always chooses the next step proportional to the weights on the incident edges. See \cite[Chapters 6.1 and 6.2]{LyonsBook} for details.
\begin{thm}
    A network on an infinite locally finite graph $G$ is amenable if and only if the spectral radius of the network random walk $\rho(G)$ is $1$.
\end{thm}

We apply this theory to an explicit random walk whose bias accounts for the weight function $\w$. Consider a $\Gamma$ closed subgroup of $\Aut(G)$ that acts transitively on $G$ and let $\w^o(\cdot)$ be the induced relative weight function on the vertices as in \eqref{def:haarweights} with respect to some reference point $o\in V$. Besides edge weights being constant $1$, the most natural network to be associated with $(G,\w^o)$ is the one whose stationary measure is $\w^o$. This is the network where the weight of edge $(x,y)$ are $\widetilde\w^o (x,y):=\sqrt{\w^o(x)\w^o(y)}$ (or any multiple of it by a positive constant). Indeed, letting the transition probability from $x$ to $y$ be $p^{(\w)}(x,y)\propto \widetilde\w^o (x,y)$, we have reversibility:

\begin{align*}
\w^o(x)p^{(\w)}(x,y)&=\w^o(x)\frac{\sqrt{\w^o(x)\w^o(y)}}{\sum_{z\sim x} \sqrt{\w^o(x)\w^o(z)}}=\frac{\sqrt{\w^o(x)\w^o(y)}}{\sum_{z\sim x} \sqrt{\frac{\w^o(z)}{\w^o(x)}}}\\
&=\frac{\sqrt{\w^o(x)\w^o(y)}}{\sum_{z\sim y} \sqrt{\frac{\w^o(z)}{\w^o(y)}}}=\w^o(y)\frac{\sqrt{\w^o(x)\w^o(y)}}{\sum_{z\sim y} \sqrt{\w^o(y)\w^o(z)}}
=\w^o(y)p^{(\w)}(y,x),    
\end{align*}
where in the first equality in the second line we used transitivity of $G$ and $\Gamma$-invariance of $\w$.

The corresponding random walk is called the \textbf{square-root biased} ($\sqrt{\w}$-biased) random walk $X^{(\w)}_n$ on $G$ and was first used by Tang in \cite{Pengfei18} to prove indistinguishability of heavy clusters of Bernoulli percolation. The name comes from the fact that one can reinterpret the transition probabilities as at each step the random walk makes a step to a neighboring vertex with probability proportional to the square root of its weight: for any $x\sim y$
\begin{equation}\label{eq:biaskernel}
p^{(\w)}(x,y)=\frac{\sqrt{\w^o(x)\w^o(y)}}{\sum_{z\sim x} \sqrt{\w^o(x)\w^o(z)}}=\frac{\sqrt{\w^o(y)}}{\sum_{z\sim x}\sqrt{\w^o(z)}}=\frac{1}{D^{\w}}\sqrt{\w^x(y)},
\end{equation}
where we call the normalization constant the \textbf{$\sqrt{\w}$-degree} of $G$ as 
    \begin{equation}\label{eq:D}
    D^{\w}:=D^{\w}(o,G):=\sum_{z\sim o}\sqrt{\w^o(z)}.
    \end{equation}
Notice that by transitivity of $G$ and the invariance of the weight function $\w$, $D^{\w}$ does not depend on the choice of the reference vertex $o$. Also note that when $G$ is transitive and unimodular (\ie\ $\w\equiv1$) it reduces to the simple random walk on $G$.

Following \cite[Chapter 6.1]{LyonsBook} with conductance given by 
\begin{equation}\label{def:cond}
    c(x,y):=\frac{\widetilde\w^o (x,y)}{D^{\w}}=\w^o(x)p^{(\w)}(x,y)
\end{equation}
we now define the \textbf{weighted edge-expansion constant} for a subgraph $F\subseteq G$ as
    \begin{equation}\label{eq:edgewCheegerset}
    \Phi^{\w}_E(F;G,c):=\frac{c(\partial_E F)}{\w^o(F)}=\frac{1}{\w^o(F)}\sum_{x\in F\atop (x,y)\in \partial_E F} \frac{\sqrt{\w^o(x)\w^o(y)}}{D^{\w}},
    \end{equation}
    where for a set of edges $E'\subset E$ we set $c(E'):=\sum_{e\in E'}c(e)$ and $\partial_E F$ denotes the external edge boundary. Taking infimum over all subgraphs of finite weight we define
    \begin{equation}\label{eq:edgewCheeger}
    \Phi^{\w}_E(G,c):=\inf_{F\subseteq V\atop \w(F)<\infty}\Phi^{\w}_E(F;G,c).
    \end{equation}

\begin{lem}\label{lem:edgeamenability}
    Let $\Gamma$ be a closed subgroup of $\Aut(G)$ that acts transitively on $G$ and let $\w$ be the induced relative weight function as in \eqref{def:haarweights}. Suppose each vertex in $G$ is $d$-regular then
    \[
    \frac{1}{\sqrt{d} D^{\w}}\cdot \iota_E^{\,\w}(G)\le\Phi^{\w}_E(G,c)\le\iota_E^{\,\w}(G)
    \]
    and
    \[
   \frac{1}{\sqrt{d} D^{\w}} \cdot \Phi^{\w}_V(G)\le \Phi^{\w}_E(G,c)\le \frac{d}{D^{\w}}\cdot \Phi^{\w}_V(G),
    \]
    where $\Phi^{\w}_V(G)$ is as in \eqref{eq:wCheeger} and $\iota_E^{\,\w}(G)$ is as in \eqref{def:BLPS_edgeiso}.
    
    In particular, a graph $G$ is $\w$-amenable if and only if $\Phi^{\w}_E(G,c)=0$.
\end{lem}
\begin{proof}
Indeed, to derive the first set of inequalities notice that
Recall that that for any pair of adjacent vertices $x,y$ we have that $1/d\le\w^x(y)\le d$. Let $F\subseteq G$ be an arbitrary finite subgraph then to see the first set of inequalities notice that
\begin{align*}
    c(\partial_E F)&=\sum_{x\in F\,,\,y\in\partial_V F \atop x\sim y}  \w^o(x)p^{(\w)}(x,y)\le \sum_{x\in F\,,\,y\in\partial_V F \atop x\sim y}  \w^o(x)
\end{align*}
and by rewriting $p^{(\w)}(x,y)$ as in \eqref{eq:biaskernel} we get
\begin{align*}
    c(\partial_E F)&=\sum_{x\in F\,,\,y\in\partial_V F \atop x\sim y}  \w^o(x)\frac{\sqrt{\w^{x}(y)}}{D^{\w}}\ge \frac{1}{\sqrt{d} D^{\w}} \sum_{x\in F\,,\,y\in\partial_V F \atop x\sim y}  \w^o(x).
\end{align*}
To derive the second set of inequalities we use reversibility
    \begin{align*}
    c(\partial_E F)&=\sum_{x\in F\,,\,y\in\partial_V F \atop x\sim y}  \w^o(x)p^{(\w)}(x,y)=\sum_{x\in F\,,\,y\in\partial_V F \atop x\sim y}  \w^o(y)p^{(\w)}(y,x)\\
    &\le\sum_{x\in F\,,\,y\in\partial_V F \atop x\sim y}  \w^o(y)\le d \cdot \w^o(\partial_V F)
    \end{align*}
    and again \eqref{eq:biaskernel} to conclude that
    \begin{align*}
    c(\partial_E F)
     &=\sum_{x\in F\,,\,y\in\partial_V F \atop x\sim y}  \w^o(y)p^{(\w)}(y,x)\ge \frac{1}{\sqrt{d} D^{\w}}  \sum_{y\in\partial_V F}\w^o(y)= \frac{1}{\sqrt{d} D^{\w}}  \cdot \w^o(\partial_V F).\tag*{\qedhere}
    \end{align*}
\end{proof}
Lemma~\ref{lem:edgeamenability} together with a simple application of \cite[Theorem 6.7]{LyonsBook} yields the following analog of Kesten's characterization.
\begin{thm}\label{thm:wKesten}
    Let $\Gamma$ be a closed subgroup of $\Aut(G)$ that acts transitively on $G$ and let $\w$ be the induced relative weight function as in \eqref{def:haarweights}. Then the spectral radius $\rho^{(\w)}(G)$ of $X^{(\w)}$ satisfies
    \[
    \frac12\Phi^{\w}_E(G,c)^2\le1-\sqrt{1-\Phi^{\w}_E(G,c)^2}\le1-\rho^{(\w)}(G)\le \Phi^{\w}_E(G,c).
    \]
    In particular, $\rho^{(\w)}(G)=1$ if and only if $G$ is $\w$-amenable.
\end{thm}

By the exact same arguments as in the proofs of Theorem~\ref{thm:Ewh{D}inv} and Corollary~\ref{thm:ph_threshold} (replacing $\wh{A}^{\,\w}$ from with $A^{\w}(F):=\frac{1}{\w^o(F)}\sum_{x,y\in F\atop (x,y)\in F}\sqrt{\w^o(x)\w^o(y)}/D^{\w}$) one can get percolation threshold results for this notion of the weighted degree. First, similarly to the definition of weighted degree of $x$ in $F$ as in \eqref{def:BLPS_D_F} we define the $\sqrt{\w}$-\textbf{degree of a vertex $x$ in a subgraph} $F\subseteq G$ as
     \begin{equation}\label{def:sqrtdeg}
     D^{\w}(x,F):=\sum_{z\in V\atop (x,z)\in F}\sqrt{\w^x(z)}.
     \end{equation}
     \begin{thm}\label{thm:EDinv}
     Let $\Gamma$ be a closed subgroup of $\Aut(G)$ that acts transitively on $G$ and $\w$ be the induced relative weight function as in \eqref{def:haarweights}. Suppose that $\mathbf{P}$ is a $\Gamma$-invariant bond percolation in $G$. If $\mathbf{P}$-a.s.\  all clusters are light, then
     \begin{equation*}
     \mathbf{E} D^{\w}(x,\go) \leq (1-\Phi_E^{\w}(G,c)) D^{\w}.   
     \end{equation*}
\end{thm}
\begin{cor}\label{thm:ph_threshold}
     Let $\Gamma$ be a closed subgroup of $\Aut(G)$ that acts transitively on $G$ and $\w$ be the induced relative weight function as in \eqref{def:haarweights}. Suppose $\mathbf{P}$ is a $\Gamma$-invariant bond percolation on $G$. If 
     \[
     \frac{\mathbf{E} D^{\w}(x,\go)}{D^{\w}}>1-\Phi_E^{\w}(G,c)
     \] then with positive probability there is a heavy cluster in $\go$. Moreover, with positive probability there is a cluster $C\subseteq \go$ such that $p_h(C,\gC)<1$.  
\end{cor}
\begin{lem}\label{lem:weightdegineq}
    Let $\Gamma$ be a closed subgroup of $\Aut(G)$ that acts transitively on $G$ and let $\w$ be the induced relative weight function as in \eqref{def:haarweights}. Assume that $G$ is $d$-regular then $
    D^{\w}\le d,
    $
    with equality attained if and only if $\Gamma$ is unimodular.
\end{lem}
\begin{proof}
    Suppose $\delta>1$ such that there is a vertex $z\in V$ that is a neighbor of $o$ and $\w^o(z)=\delta$. Recall the notation $N_{\delta}:=\abs{\{z\sim o \mid \w^o(z)=\delta\}}$
    and that
    \begin{equation}\label{eq:ratioofneighb}
        \delta N_\delta=N_{1/\delta}.
    \end{equation}
    Using this we can rewrite the difference between the degree and $\sqrt{\w}$-degree of $G$ as follows:
    \begin{align*}
        d-D^{\w}&=\sum_{z\sim o}1-\sqrt{\w^o(z)}=\sum_{\delta>1} \left(1-\sqrt{\delta}\right) N_{\delta}+\left(1-\frac{1}{\sqrt{\delta}}\right)N_{1/\delta}\\
        &=\sum_{\delta>1} \left(\delta+1-2\sqrt{\delta}\right)N_{\delta}>0.\tag*{\qedhere}
    \end{align*}
\end{proof}

\begin{lem}\label{lem:symwalk}
     Let $\Gamma$ be a closed subgroup of $\Aut(G)$ that acts transitively on $G$ and let $\w$ be the induced relative weight function as in \eqref{def:haarweights}. Suppose $X^{(\w)}_n$ is the $\sqrt{\w}$-biased random walk on $G$. Then $Y_n:=\log\left(\w^o\left(X^{\w}_n\right)\right)$ is a symmetric random walk on a countable subset of $\bR$.
\end{lem}
\begin{proof}
    The statement follows easily from the following computation.
    The probability that $\sqrt{\w}$-biased random walk makes a step from $o$ to the level with weight $\delta$ can be written as
    \begin{align*}
    \pr_o \left(\w^o\left(X^{(\w)}_1\right)=\delta\right)&=\sum_{z\sim o\atop \w^o(z)=\delta} p^{(\w)}(o,z)=\frac{1}{D^{\w}}\sum_{z\sim o\atop \w^o(z)=\delta}\sqrt{\delta}=\frac{\sqrt{\delta}}{D^{\w}}N_\delta\\
    [\textnormal{ by \eqref{eq:ratioofneighb} }]&=\frac{1}{\sqrt{\delta} D^{\w}}N_{1/\delta}=\frac{1}{D^{\w}}\sum_{v\sim o\atop \w^o(v)=\delta^{-1}}\frac{1}{\sqrt{\delta}}=\sum_{v\sim o\atop \w^o(v)=\delta^{-1}} p^{(\w)}(o,v)\\
    &=\pr_o \left(\w^o\left(X^{(\w)}_1\right)=\frac{1}{\delta}\right).\tag*{\qedhere}
    \end{align*}
\end{proof}

\begin{lem}\label{lem:specrad}
    Let $\Gamma$ be a closed subgroup of $\Aut(G)$ that acts transitively on $G$ and let $\w$ be the induced relative weight function as in \eqref{def:haarweights}. Suppose $\rho(G)$ and $\rho^{\w}(G)$ are the spectral radii of the simple random walk on $G$ and the $\sqrt{\w}$-biased random walk on $G$, respectively. Then
    \[
    \rho(G)\le \rho^{\w}(G).
    \]
\end{lem}
\begin{proof}
    Recall that the level of a vertex $x\in V$ is defined as $$\level(x):=\{y\in V \mid \w^o(y)=1\}.$$ In particular, levels do not depend on the choice of the reference point $o$. Notice that when a simple random walk makes a step from $x\in V$ to a vertex in level $L$ such a vertex is chosen uniformly at random among the neighbors of $x$ that belong to the level $L$, moreover the same is true for the $\sqrt{\w}$-biased random walk. 

    The symmetric random walk $Y_n:=\log\left(\w^o\left(X^{\w}_n\right)\right)$ from Lemma~\ref{lem:symwalk} has greater return probabilities than the one generated by a simple random walk $X_n$ (that is, $\log\left(\w^o\left(X_n\right)\right)$).
    Writing the event that a random walk returned to $o$ as the event it came back to the original level and then choosing $o$ in it from a corresponding distribution, we notice that $\sqrt{\w}$-biasing makes the return probability to the level bigger while keeping the latter distribution unchanged. Thus
    \begin{align*}
        \rho(G)&=\lim\sup_{n\to \infty} p_{2n}(o,o)^{\frac{1}{2n}}\le\sup_{n\to \infty} p_{2n}^{\w}(o,o)^{\frac{1}{2n}}=\rho^{\w}(G).\tag*{\qedhere}
    \end{align*}
\end{proof}

%
\section{Level-amenability and hyperfiniteness}\label{sec:levelamen}
%
Besides weighted-amenability, there are other natural candidates to grasp ``amenable behavior'' of nonunimodular transitive graphs. We define two of these alternatives next. The first one was implicit in \cite[Theorem 5.1]{BLPS99inv}, where it was essentially shown to be equivalent to $\w$-amenability.

\begin{defn}[Hyperfiniteness]\label{def:hyperfinite}
Let $\Gamma$ be a closed subgroup of $\Aut(G)$ that acts transitively on $G$ and $\w$ be the induced relative weight function as in \eqref{def:haarweights}. Suppose $\go$ is an invariant random subgraph of $G$. 
We say that $\go$ is \textbf{hyperfinite} if there is an increasing family random subgraphs $\go_1\subseteq\go_2\subseteq\ldots\subseteq\go_n\subseteq\ldots\subseteq\go$ such that 
\begin{enumerate}
    \item the joint law of $(\go,\{\go_n\}_{n\in\bN})$ is $\Gamma$-invariant,
    \item\label{def:hyperfinite1} $\cup\omega_n=\go$,
    \item\label{def:hyperfinite2} for all $n\ge1$ all clusters are finite $\omega_n$-a.s.\ 
\end{enumerate}

\end{defn}
We will write $\go_n\nearrow \go$ if $\go_n\subseteq\go_{n+1}$ for all $n\in\bN$ and $\cup_n \go_n=\go$.

The above definition can be viewed as an adaptation of the notions of $\mu$-hyperfiniteness for graphings (or for measurable equivalence relations) and unimodular random graphs \cite{URG}.

\begin{lem}\label{lem:finite->light}
    Let $\Gamma$ be a closed subgroup of $\Aut(G)$ that acts transitively on $G$ and $\w$ be the induced relative weight function as in \eqref{def:haarweights}. Then
    \begin{enumerate}
        \item\label{lem:finite->light:wamen} In Definition~\ref{def:wamen} one can replace $``\abs{F}<\infty"$  by $``\w^o(F)<\infty"$, \ie\
        $$
            \quad\Phi^{\w}_V(G):=\inf_{F\subseteq V\atop \abs{F}<\infty}\frac{\w^o(\partial_V F)}{\w^o(F)}=0\quad\textnormal{if and only if}\quad\inf_{F\subseteq V\atop \w^o(F)<\infty}\frac{\w^o(\partial_V F)}{\w^o(F)}=0.
        $$
        \item\label{lem:finite->light:hf} Definition~\ref{def:hyperfinite} remains the same if we replace the second condition by
        \[
        \textnormal{for all $n\ge1$ all clusters are \textit{light} $\omega_n$-a.s.}
        \]
    \end{enumerate}
\end{lem}
\begin{proof}
    Since every finite subgraph is automatically light the forward direction of \eqref{lem:finite->light:wamen} is trivial. For the converse direction suppose $F_n\subseteq G$ is a sequence of light subgraphs satisfying $\w^o(\partial_VF_n)\le \eps_n \w^o(F_n)$, for some $\eps_n\searrow0$. Since for every $n$ we have that $\sum_{x\in F_n}\w^o(x)<\infty$ there is a large enough finite set $F'_n$ such that
    $\sum_{x\in F_n\setminus F'_n}\w^o(x)=\delta_n$, for an arbitrarily small $\delta_n\searrow0$. Notice that $\w^o(\partial_VF'_n)\le\w^o(\partial_VF_n) +\delta_n$. Thus,
    \[
    \frac{\w^o(\partial_VF'_n)}{\w^o(F'_n)}\le \frac{\w^o(\partial_VF) +\delta_n}{\w^o(F_n)-\delta_n}\le \frac{\eps_n\w^o(F_n) +\delta_n}{\w^o(F_n)-\delta_n}\to 0 \textnormal{ as } n\to\infty.
    \]
    We now show \eqref{lem:finite->light:hf}. Again, the forward direction is trivial. Suppose now that $\{\go_n\}_{n\in\bN}$ is an increasing family of configurations sampled from $\{\textbf{P}_n\}_{n\in\bN}$ such that $\go_n$ contains only light clusters for all $n\ge1$. Since each light cluster $C$ admits finitely many vertices of maximal weight, choose one of such vertices uniformly at random and denote it by $z_C$. Construct a sequence of bond percolation configurations $\eta_m(C)$ by
    \[
    (x,y)\in \eta_m(C) \Leftrightarrow (x,y)\in C \textnormal{ and } \max\left\{\dist_C(x,z_C),\dist_C(y,z_C)\right\}\le m.
    \]
    By local finiteness of $G$ for all $m\in\bN$ all clusters in $\eta_m(C)$ are finite and $\eta_m(C)\nearrow C$. Notice that $\eta_m(\go_n):=\bigcup_C \eta_m(C)$ witnesses hyperfiniteness of $\go_n$. The conclusion now follows from the fact that the countable increasing union of hyperfinite graphs is hyperfinite.
\end{proof}

This notion of hyperfiniteness for a random subgraph of $G$ is connected to hyperfiniteness from measure group theory via the cluster graphing construction (see Subsection~\ref{sec:mgt}).

\begin{lem}\label{lem:hf=mu_hf}
    Let $\Gamma$ be a closed subgroup of $\Aut(G)$ that acts transitively on $G$ and $\w$ be the induced relative weight function as in \eqref{def:haarweights}. Suppose $\go$ is an invariant random subgraph of $G$. Then $\go$ is hyperfinite if and only if $\go$ admits a cluster graphing $\cG^{\mathrm{cl}}$ on a standard probability space $(Y,\nu)$ that is $\nu$-hyperfinite.
\end{lem}
\begin{proof}
    $\Rightarrow:$ Let $\go_n$ be a sequence of random subgraphs that witnesses hyperfiniteness of $\go.$ Let $\mathbf{P}$ denote the joint law of $(\go,\{\go_n\}_{n\in\bN})$ and consider the \textit{joint cluster graphing}, that is the same construction as in the summary at the end Subsection~\ref{sec:mgt}, but started with an ergodic free pmp action of $\Gamma$ on the standard
    probability space $(X,\mu)$ that factors onto $\left((2^E)^\bN, \mathbf{P}\right)$. This yields an increasing family of cluster graphings $\cG^{\mathrm{cl}}_{\go_n}$ on $(Y,\nu)$ with finite connected components on a $\nu$-conull set (because each of these connected components is isomorphic to a cluster in the corresponding configuration $\go_n$), and $\cG^{\mathrm{cl}}_{\go}=\bigcup_n\cG^{\mathrm{cl}}_{\go_n}$. Thus $\cG^{\mathrm{cl}}_{\go}$ is $\nu$-hyperfinite.
    
    $\Leftarrow:$ Suppose the cluster graphing $\cG^{\mathrm{cl}}_{\go}$ of $\go$ is an increasing union of Borel graphs $\{\cG_n\}_{n\in\bN}$ whose connected components are finite. By Observation~\ref{obs:cluster} each $\cG^{\mathrm{cl}}_{\go}$-connected component of $\wt{x}=[(x,o)]_\Gamma\in Y$ is isomorphic to the connected component of $o$ in $\pi(x)$, thus the images under this isomorphism of the $\cG_n$ yield an increasing exhausting family of invariant random  subgraphs of $\go$. In particular, $\go$ is hyperfinite.
\end{proof}

To define level-amenability of a graph we first introduce some terminology. Recall that the level of a vertex $x\in V$ is defined as $
\level(x):=\{y\in V \mid \w^x(y)=1\}.$ A subgraph induced by a finite union of levels is called a \textbf{slice}. The idea to characterize the behavior of percolation on nonunimodular graphs by the existence of nonamenable slices first appeared in \cite{Timar06nonu}. We formalize it here via the following definition.

\begin{defn}
    We say that $G$ is \textbf{level-amenable} if any slice of $G$ is amenable, \ie~every connected component in the slice is an amenable graph.
\end{defn}
Notice that $\w^o(\cdot)$ restricted to a slice takes only finitely many values, thus each slice is a quasi-transitive \textit{unimodular} graph (see \cite[Proposition 4.3]{Pengfei18} for a proof).

\begin{obs}\label{obs:levels}
    Let $\Gamma$ be a closed subgroup of $\Aut(G)$ that acts transitively on $G$ and $\w$ be the induced relative weight function as in \eqref{def:haarweights}. Define the \textbf{level graph} $\cL(G)$ as the graph on the set of levels of $G$ with edges induced by the adjacency in $G$. Then $\cL$ is the Cayley graph of a finitely generated Abelian group. In particular, it is hyperfinite.
\end{obs}

Given an invariant random partition of the set of levels of $G$, we define a \textbf{level-percolation} on $G$ is defined on by retaining only those edges whose endpoints are in levels in the same class in the partition. 
A hyperfinite exhaustion of the level graph $\cL(G)$ defines an increasing family of invariant random partitions such that each class is a slice and their union is $G$. We are now ready to prove Theorem \ref{thm:levelamen-hf}.

\begin{proof}[Proof of Theorem~\ref{thm:levelamen-hf}]
As we mentioned above $\eqref{thm:levelamen-hf:amen}\Leftrightarrow\eqref{thm:levelamen-hf:hf}$ was essentially shown in \cite[Theorem 5.1]{BLPS99inv} (see Theorem~\ref{thm:BLPShf}), in fact, their proof gives that a $\w$-amenable graph $G$ admits a factor of iid hyperfinite exhaustion.  In the context of measured group theory the equivalence of $\mu$-hyperfiniteness and $\mu$-amenability is classical \cite{CFW}, including the  characterization via the weighted isoperimetric constant \cite{kaimanovich1997amenability} for bounded degree Borel graphs.

$\eqref{thm:levelamen-hf:amen}\Rightarrow\eqref{thm:levelamen-hf:lvl}:$ Fix $\eps>0$, suppose $F\subset G$ is of finite weight and is such that $\w^{o}(\partial_VF)/\w^{o}(F)<\eps.$ Consider a slice $S$ an let $h$ be the number of levels contained in it. By transitivity of $G$ there is a family of automorphisms $\gamma_i\in\gC$ such that for any two levels there is exactly one $\gamma_i$ that maps the first one to the second one.
The union of $\gamma_i S$ covers $G$ by copies of $S$ such that every level of $G$ is contained in exactly $h$ many copies of $S$. Denote $F_i:=F\cap \gamma_i S$ and notice that 
    \[
    \sum_{i}\frac{\w((\partial_V F_i) \cap \gamma_i S)}{\w(F_i)}\cdot\frac{\w(F_i)}{h\w(F)}=\sum_{i}\frac{\w((\partial_V F_i) \cap \gamma_i S)}{h\w(F)}\le \frac{\w (\partial_V F)}{\w(F)}<\eps.
    \]
Now since $\sum_{i}\w(F_i)/(h\w(F))=1$ we conclude that there is some $i$ such that 
    \[
    \frac{\w((\partial_V F_i) \cap \gamma_i S)}{\w(F_i)}<\eps.
    \]
Thus the same holds for $\gamma_i^{-1}F_i\subset S$. Taking arbitrarily small $\eps$ yields that $S$ is $\w$-amenable. Since $S$ is unimodular, it is then also amenable.
    
$\eqref{thm:levelamen-hf:lvl}\Rightarrow\eqref{thm:levelamen-hf:hf}:$ Consider an increasing sequence of level-percolations on $G$ defined by a hyperfinite exhaustion of the level graph $\cL(G)$. Each such percolation splits $G$ into a disjoint union of slices, all of which are amenable unimodular quasi-transitive graphs. By Theorem~\ref{thm:BLPShf} they are hyperfinite. Since $G$ is the increasing union of hyperfinite graphs it is also hyperfinite.


$\eqref{thm:levelamen-hf:hf}\Rightarrow\eqref{thm:levelamen-hf:amen}:$ Suppose $G$ is hyperfinite, as witnessed by a sequence of processes $\{\textbf{P}_n\}_{n\in\bN}$. Let $\{\go_n\}_{n\in\bN}$ be configurations sampled from $\{\textbf{P}_n\}_{n\in\bN}$. Since $\{\textbf{P}_n\}_{n\in\bN}$ are stochastically increasing, we may assume that $\go_n\nearrow G$. Without loss of generality, we may assume that for any neighboring vertices $u$ and $v$ in $G$ that are also in the same cluster $\go_n$, we have that the edge $(u,v)\in \go_n$, because after imposing such a requirement percolation remains $\Gamma$-invariant and the components of the $\go_n$ have the same, light vertex sets. Thus, denoting the expectation with respect to $\textbf{P}_n$ by $\textbf{E}_n$, we have for any $x\in V$
     \[
     \textbf{E}_{n} \wh{D}^{\w}(x,\go_n)\to \wh{D}^{\w}(x,G)=d \quad\textnormal{as}\quad n\to\infty,
     \]
     where $\wh{D}(\cdot,\cdot)$ denotes the the weighted degree as in \eqref{def:BLPS_D_F} and $d$ is the degree in $G$. 
     The conclusion now follows, from Theorem~\ref{thm:Ewh{D}inv}. 
\end{proof}
\begin{cor}\label{cor:wamen-treeable}
Let $\Gamma$ be a closed subgroup of $\Aut(G)$ that acts transitively on $G$ and $\w$ be the induced relative weight function as in \eqref{def:haarweights}. Suppose $G$ is $\w$-amenable. Then $G$ admits a $\Gamma$-invariant random spanning tree.
\end{cor}
\begin{proof}
    By Theorem~\ref{thm:levelamen-hf}, $G$ is hyperfinite. Consider a sequence of measures $\{\textbf{P}_n\}_{n\in\bN}$ that witnesses hyperfiniteness, and an increasing percolation configurations $\go_n\nearrow G$ sampled from them, such that every cluster of $\go_n$ is finite.
We build the spanning tree inductively. Assuming that every cluster of $\go_k$ is spanned by a tree, we extend it by spanning each cluster of $\go_{k+1}$ independently of each other. For each such cluster $C\in \go_{k+1}$ sequentially choose an edge from $C\setminus \go_{k}$ uniformly at random and add it if it does not create a cycle until $C$ is spanned by a tree. Since the $\{\textbf{P}_n\}_{n\in\bN}$ are $\Gamma$-invariant and we used only uniform choice conditioned on the sample from such measures, thus the resulting tree also has a $\Gamma$-invariant law.
\end{proof}

The following easy corollary will be handy in a later construction.
\begin{cor}\label{cor:inv-slice}
Let $\Gamma$ be a closed subgroup of $\Aut(G)$ that acts transitively on $G$ and $\w$ be the induced relative weight function as in \eqref{def:haarweights}. If $G$ is $\w$-nonamenable, then there is an invariant level-percolation on $G$ that a.s.\ contains a connected component that is a nonamenable quasi-transitive unimodular graph.
\end{cor}

\section{Invariant random spanning trees and forests}\label{sec:invarforest}

An essential tool for many percolation results in the unimodular setting is \cite[Proposition 7.2]{BLPS99inv}. It connects the number of ends of an invariant random spanning forest $\cF$ to $p_c(\cF)$, and to the average degree in the infinite components of $\cF$. In addition to their applications in percolation theory,
invariant spanning forests also played an important role in measured group theory, especially in the pmp setting. Famously, this machinery enabled Gaboriau and Lyons to prove a positive answer to a measure theoretic version of the Day--von Neumann conjecture \cite{Gaboriau-Lyons}. As we highlighted in Subsection~\ref{sec:mgt}, extensions of this theory to the mcp setting have received an increasing amount of attention in recent years. In this section, we consider invariant random spanning trees and forests in the context of $\w$-amenability and nonunimodular transitive graphs.

First, we recall that the \textbf{Free Minimal Spanning Forest} on the graph $G$ is a random subforest that is constructed as follows:

\begin{itemize}
    \item Let $\{U_e\}_{e\in E}$ be a collection of independent random variables with $\mathrm{Uniform}[0,1]$ distribution. Notice that almost surely we have $U_e\neq U_{e'}$ for each pair of distinct edges $e$ and $e'$.
    
    \item For each cycle in $G$, delete the edge $e$ with the largest value of the label $U_e$. 
\end{itemize}
In other words, for each $e \in E(G)$, we have $e \in \mathrm{FMSF}(G)$ if and only if each cycle containing $e$ also contains another edge $e'$ with $U_e < U_{e'}$.

In \cite{FmaxSF22} the authors introduced a weighted analog of the Free Minimal Spanning Forest (FMSF) called the Free $\w$-Maximal Spanning Forest, denoted as $\fmax_{\w}$, and showed that it naturally extends many desired properties of FMSF to the nonunimodular setting. 
Given a relative weight function $\w^o:V\to\bR^+$ with respect to some reference vertex $o\in V$ the \textbf{Free $\w$-Maximal Spanning Forest} $\fmax_{\w}(G)$ is constructed in a similar fashion:

\begin{itemize}
    \item For each edge $e=(u,v)\in E(G)$ define $\w^o(e)=\min(\w^o(u),\w^o(v))$.
    \item Let $\{U_e\}_{e\in E}$ as above.
    \item For each cycle in $G$, delete the edge $e$ with the lowest weight $\w^o(e)$, using $U_e$ as a tie-breaker (deleting the edge with the largest value).
\end{itemize}

In other words, for every cycle in $G$ select the set of edges that contain vertices with the smallest $\w$-weight and delete the one with the largest $U_e$ label. Notice that when $G$ is unimodular we have that $\fmax_{\w}(G) = \FMSF(G)$. 

First we would like to highlight some of the results from \cite{FmaxSF22}. We present the first one for Bernoulli$(p)$ percolation, however, the original result holds more generally for insertion and deletion tolerant $\Gamma$-invariant percolation processes see \cite[Theorem 4.16 and Remark 4.17]{FmaxSF22}.

\begin{thm}[{\cite{FmaxSF22}}]\label{thm:CorBernPerc}
Let $\Gamma$ be a closed subgroup of $\Aut(G)$ that acts transitively on $G$ and $\w$ be the induced relative weight function as in \eqref{def:haarweights}. Suppose $p\in[0,1]$ is such that there are infinitely many heavy clusters $\pr_p$-a.s. Then for $\pr_p$-a.e.\ configuration $\omega$, for every heavy cluster $C \subseteq\omega$, the random forest $\fmax_{\w}(C)\subseteq\fmax_{\w}(\omega)$ a.s.\ contains a tree $T \subseteq C$ with infinitely many $\w$-nonvanishing ends, none of which is isolated.
\end{thm}


\begin{thm}[{\cite[Corollary 5.9]{FmaxSF22}}]\label{thm:cluster-graphing_nonamenable}
Let $\Gamma$ be a closed subgroup of $\Aut(G)$ that acts transitively on $G$ and $\w$ be the induced relative weight function as in \eqref{def:haarweights}.
Let $\mathbf{P}$ be a $\Gamma$-invariant percolation such that with positive probability there a cluster with $\ge 3$ $\w$-nonvanishing ends.
Then the cluster graphing $\cG^{\mathrm{cl}}$ with the measure induced by $\mathbf{P}$ is $\mu$-nonamenable.
In fact, it contains a $\mu$-nonamenable Borel subforest.
\end{thm}


The following theorem is largely an adaptation of the results due to Adams and Lyons \cite{AdamsLyons} (see also \cite[Section 3.6]{JKLcber}) and \cite{AnushRobin}, as in Theorem~\ref{thm: Anush-Robin}, to our setting. Our main addition to the statement is the third condition that is similar to the $p_c$ condition from \cite[Proposition 7.2]{BLPS99inv}.

\begin{thm}\label{thm:trees and ph}
    Let $\Gamma$ be a closed subgroup of $\Aut(G)$ that acts transitively on $G$ and $\w$ be the induced relative weight function as in \eqref{def:haarweights}. Suppose $\cF$ is a $\Gamma$-invariant random spanning subforest of $G$. Then almost surely the following are equivalent:
    \begin{enumerate}
        \item\label{thm:trees and ph:hf} $\cF$ is hyperfinite.
        \item\label{thm:trees and ph:endselection} On a possibly larger probability space there is an invariant way to select an end in every infinite tree in $\cF$. 
        \item\label{thm:trees and ph:ph} For every heavy tree $T\in \cF$ we have $p_h(T)=1$ a.s.
        \item\label{thm:trees and ph:ends}  Every heavy tree $T\in \cF$ has $\le2$ $\w$-nonvanishing ends.
        \item\label{thm:trees and ph:geod}  Every heavy tree $T\in \cF$ has $\le2$ ends with heavy geodesics.
    \end{enumerate}
\end{thm}
\begin{rem}
    Without enlarging the probability space \eqref{thm:trees and ph:endselection} can be changed to the existence of an invariant selection at most two ends or finitely many vertices in every tree in $\cF$.
\end{rem}
\begin{proof}[Proof of Theorem~\ref{thm:trees and ph}]
    \eqref{thm:trees and ph:hf}$\Rightarrow$\eqref{thm:trees and ph:endselection}: First we consider light infinite trees in $\cF$. Since every such cluster has finitely many vertices of maximal weight, choosing one
    of them at random allows for an invariant selection of a vertex. Then we can select an invariant
    random end from all the ends by choosing a random infinite path that starts at this vertex, by
    iteratively choosing its $n$’th vertex uniformly at random from all vertices that allow a continuation
    to an infinite path. Thus, it remains to consider heavy trees. 
    
    Consider a cluster graphing $\cG^{\mathrm{cl}}$ of $\cF$. Then by Lemma~\ref{lem:hf=mu_hf} $\cF$ is hyperfinite if and only if $\cG^{\mathrm{cl}}$ is $\nu$-hyperfinite. Now, by \cite[Lemma 2.21]{JKLcber}, in a.e.~$\cG^{\mathrm{cl}}$-connected component we may select $\le 2$-ends or finitely many points in a Borel way. By Observation~\ref{obs:cluster}, $\cG^{\mathrm{cl}}$ components are isomorphic to corresponding clusters of $\cF$. Hence the selection in $\cG^{\mathrm{cl}}$ induces a selection in the clusters of $\cF$. Clearly, by the TMTP a.s.~in every heavy tree in $\cF$ we selected $1$ or $2$ ends. In case of one we are done, while in case of two choosing between selected ends at random yields the desired selection.

    \eqref{thm:trees and ph:endselection}$\Rightarrow$\eqref{thm:trees and ph:ph} 
    Suppose every heavy tree $T\in \cF$ has exactly one distinguished end, call it $\xi(T)$, and consider Bernoulli$(p)$ percolation on $\cF$ for any $p<1$. Since for every point in $T$ the geodesic path towards $\xi(T)$ will contain closed edges infinitely often, every cluster in $T$ will have a unique closest point to $\xi(T)$. By the TMTP all clusters must be light, hence $p_h(T)=1$. 
    
    \eqref{thm:trees and ph:ph}$\Rightarrow$\eqref{thm:trees and ph:hf}.
    Assume that for every heavy tree $T\in\cF$ we have $p_h(T)=1$. Associate an independent Uniform $[0,1]$ random variable $U_e$ to every edge $e\in \cF$ and for $\eps\in [0,1]$ let $\cF_{1-\eps}$ consist of those edges that have $U_e\le 1-\eps$. For every $\eps>0$ all clusters in $\cF_{1-\eps}$ are light and $\cF_{1-\eps}\nearrow T$ as $\eps\searrow0$. By part \eqref{lem:finite->light:hf} of Lemma~\ref{lem:finite->light}, $\cF$ is hyperfinite.

    \eqref{thm:trees and ph:hf}$\Leftrightarrow$\eqref{thm:trees and ph:ends},\eqref{thm:trees and ph:geod}. This equivalence follows from \cite{AnushRobin} as in Theorem~\ref{thm: Anush-Robin}. Indeed, consider the cluster graphing $\cG^{\mathrm{cl}}$ associated with $\cF$ and its connectedness relation $\cR^{\mathrm{cl}}$ that are defined on the standard probability space $(Y,\nu)$ as in Subsection~\ref{sec:mgt}. By Observation~\ref{obs:cluster} the $\cG^{\mathrm{cl}}$-connected component of $\wt{x}\in Y$ is isomorphic to the cluster of the root in the corresponding $\cF$-configuration. Moreover, this isomorphism maps values of the Radon--Nikodym cocylce $\w_\nu$ of $\cR^{\mathrm{cl}}$ with respect to $\nu$ to the relative weight function $\w$ on $G$. In particular, $\cG^{\mathrm{cl}}$ is $\nu$-hyperfinite if and only if $\cF$ is hyperfinite; 
    $\nu$-a.e.\ component $\cG^{\mathrm{cl}}$-component has $\le2$ $\w_\nu$-nonvanishing ends (resp.~$\le2$ ends with heavy geodesics) if and only if the same holds a.s.~for every heavy tree $T\in\cF$ with respect to $\w$. The conclusion now follows from Theorem~\ref{thm: Anush-Robin}.
\end{proof}

\begin{rem}[Heavy ends in a hyperfinite tree]
    One cannot replace $\w$-nonvanishing ends by heavy ends in the statement of Theorem~\ref{thm:trees and ph} and, as a consequence, in Theorems~\ref{thm:BLPSnew} and~\ref{thm:Forestequiv_new}. Indeed, consider $T_{1,2}$ from Example~\ref{ex:nonuni-tree}. Such a tree is hyperfinite, has a unique nonvanishing end, but every end in this tree is heavy. We delve deeper in the relations between various notions of ends at the end of this section.
\end{rem}

Before we move on to the proof of Theorem~\ref{thm:BLPSnew}, we observe in the next result a more general characterization of $\w$-amenability through invariant random forests.

\begin{thm}\label{thm:Forestequiv_new}
    Let $\Gamma$ be a closed subgroup of $\Aut(G)$ that acts transitively on $G$ and $\w$ be the induced relative weight function as in \eqref{def:haarweights}. Then the following are all equivalent:
    \begin{enumerate}
        \item\label{thm:Forestequiv1} $G$ is $\w$-amenable.
        \item\label{thm:Treeequiv} There is a $\Gamma$-invariant random spanning tree of $G$ with $\le 2$ $\w$-nonvanishing ends a.s.
        \item\label{thm:Treeequiv2} There is a $\Gamma$-invariant random spanning tree of $G$ with $\le 2$ ends with heavy geodesics a.s.
        \item\label{thm:Forestequiv2} For any $\Gamma$-invariant random spanning forest $\cF$ of $G$, the $\cF$-configuration $\varphi$ contains only trees with $\le 2$ $\w$-nonvanishing ends a.s.
        \item\label{thm:Forestequiv3} For any $\Gamma$-invariant random spanning forest $\cF$ of $G$, the $\cF$-configuration $\varphi$ satisfies $p_h(\varphi)=1$ a.s.
    \end{enumerate}
\end{thm}

\begin{proof}

    \eqref{thm:Forestequiv1} $\Leftrightarrow$ \eqref{thm:Treeequiv},\eqref{thm:Treeequiv2}.
    From Corollary~\ref{cor:wamen-treeable} we know that a $\w$-amenable graph $G$ admits a hyperfinite $\Gamma$-invariant random spanning tree $T$. By Theorem~\ref{thm:trees and ph} hyperfiniteness of $T$ is equivalent to having $\le 2$ $\w$-nonvanishing ends (resp.~ ends with heavy geodesics).

    \eqref{thm:Forestequiv1}$\Rightarrow$\eqref{thm:Forestequiv2}. Suppose the opposite, that there exists a $\Gamma$-invariant random spanning forest $\cF$ of $G$ that contains a component with $\ge 3$ $\w$-nonvanishing ends with positive probability. By Theorem~\ref{thm:trees and ph}, such a forest is nonhyperfinite and hence  $G$ cannot be hyperfinite. By Theorem~\ref{thm:levelamen-hf} $G$ is $\w$-nonamenable, a contradiction.
    
    \eqref{thm:Forestequiv2}$\Rightarrow$\eqref{thm:Forestequiv1}. Suppose $G$ is $\w$-nonamenable. By Corollary~\ref{cor:inv-slice} there is an invariant level-percolation on $G$ whose cluster contains a nonamenable slice a.s. By \cite[Theorem 3.10]{BLS-perturb} adopted to the quasi-transitive setting we have that such a component admits an invariant spanning forest with positive Cheeger constant and, thus, infinitely many ends. Since a slice contains only finitely many levels each these ends is $\w$-nonvanishing.
    
    \eqref{thm:Forestequiv2}$\Leftrightarrow$\eqref{thm:Forestequiv3}. Follows from Theorem~\ref{thm:trees and ph}.
\end{proof}

\begin{proof}[Proof of Theorem~\ref{thm:BLPSnew}]
\eqref{thm:BLPSnew1}$\Leftrightarrow$\eqref{thm:BLPSnew2} is part of Theorem \ref{thm:Forestequiv_new}.

\eqref{thm:BLPSnew1}$\Rightarrow$\eqref{thm:BLPSnew3} The forward direction was shown in Theorem~\ref{thm:BLPSnew}, as a $\Gamma$-invariant random spanning tree $T$ with $\le 2$ $\w$-nonvanishing ends has $p_h(T)=1$ by Theorem~\ref{thm:trees and ph}.

\eqref{thm:BLPSnew1}$\Leftarrow$\eqref{thm:BLPSnew3} The other direction follows from hyperfiniteness and is based on the first part of the argument in the proof of \cite[Theorem 8.21]{LyonsBook}. Assign to each edge $e\in E(G)$ an independent Uniform$[0,1]$ random variable $U_e$. Given a random connected subgraph $\go$ let $\go_{1-1/n}$ be its subgraph that contains all of the edges with $U_e\le1-1/n$. Since $p_h(\go)=1$ all clusters in $\go_{1-1/n}$ are light. For every vertex $x\in V$ define $W(x)$ to be the closest vertex (in $G$-distance) in $\go$ to $x$, if there are several then choose one uniformly at random. Let $\xi_n$ consist of all edges $(x,y)$ such that $W(x)$ and $W(y)$ are in the same cluster in $\go_{1-1/n}$. By the TMTP, for all $n$ every cluster in $\xi_n$ has to be light and $\xi_n\nearrow G$. Hence $G$ is hyperfinite which is equivalent to $\w$-amenability by Theorem~\ref{thm:levelamen-hf}.

\eqref{thm:BLPSnew1}$\Leftrightarrow$\eqref{thm:BLPSnew4} This was proved by Benjamini, Lyons, Peres and Schramm, as in Theorem \ref{thm:BLPSinv5.3}.
\end{proof}

\begin{lem}\label{lem:geod-path}
    Let $\Gamma$ be a closed subgroup of $\Aut(G)$ that acts transitively on $G$ and $\w$ be the induced relative weight function as in \eqref{def:haarweights}. Suppose $\mathbf{P}$ is a $\Gamma$-invariant random spanning forest such that all trees have $\le2$ $\w$-nonvanishing ends. Then $\mathbf{P}$-a.s.\ in every cluster the geodesic path is light for every vanishing end.
\end{lem}
\begin{proof}
    It  follows directly from considering the cluster graphing induced by the percolation process $\mathbf{P}$ and applying Theorem~\ref{thm: Anush-Robin} and Observation~\ref{obs:cluster}.
\end{proof}

Theorem~\ref{thm:trees and ph} and Lemma~\ref{lem:geod-path} naturally motivate the question about the relations between various notions of ends introduced in Subsection~\ref{sec:ends} at least in the context of invariant random subforests.

\begin{rem}[Nonvanishing ends, heavy ends, heavy geodesics]\label{rem:nonvan}
The results of the present section naturally raise the question on how nonvanishing ends, heavy ends and heavy geodesics are related in invariant random forests. 
As usual, let $\Gamma$ be a closed subgroup of $\Aut(G)$ that acts transitively on $G$ and $\w$ be the induced relative weight function as in \eqref{def:haarweights}, and consider some $\Gamma$-invariant random subforest of $G$.

  \textit{Case 1: $\ge 3$ (equivalently, infinitely many) nonvanishing ends.} Suppose a heavy component-tree has $\ge3$ nonvanishing ends. Then it is not hyperfinite, and (perhaps surprisingly) all vanishing ends are light, and thus have only light geodesic paths. This was shown in \cite{AnushRobin} for Borel acyclic graphs, and applying this result one can conclude the same for our setting. Geodesics towards nonvanishing ends in such trees might behave in all possible ways. For instance, consider $T_{2,3}$: it has only nonvanishing ends, and all geodesic paths that consist of only incoming (resp.~outgoing) edges in the orientation are light (resp.~heavy).

\textit{Case 2: two nonvanishing ends.} Suppose a heavy component-tree contains exactly two nonvanishing ends. Then by a simple application of TMTP all vanishing ends are light, and thus have only light geodesics. Indeed, let every vertex in such a tree send a unit mass to the closest point on the unique bi-infinite geodesic path that connects two nonvanishing ends. Similarly, the $\limsup$'s of weights along geodesic paths to nonvanishing ends must be equal to each other and greater than zero. Assuming the values of $\limsup$'s are different, orient the edges on the bi-infinite geodesic from the end with the lower value towards the one with the higher value. There must be the first vertex (in this orientation) whose weight is larger than the average of the $\limsup$'s. Since this point was chosen invariantly, this violates TMTP. Finally, assume both $\limsup$'s are zero then there are finitely many vertices on the bi-infinite path of maximal weight, this again contradicts TMTP.

  \textit{Case 3: unique nonvanishing end.} Suppose a heavy component-tree contains only one nonvanishing end. Firstly, a heavy end might be vanishing and contain only light geodesics. For example in $T_{1,2}$ every vanishing end is heavy, but the unique geodesic towards it has exponential decay in weight.
 Moreover, by Lemma~\ref{lem:geod-path} there can be at most one end with a heavy geodesic and it must be the nonvanishing end. However, there might be no ends with a heavy geodesic at all as we show in the following example.
\end{rem}
 
 \begin{ex}[Heavy forest without heavy geodesics]
     Recall the construction of $\DL(2,3)$ from Example~\ref{ex:DL}. We will construct an invariant random forest on $\DL(2,3)$ by running an invariant percolation on one of the trees it is constructed from. Start with $T_1:=T_{1,2}$ and $T_2:=T_{1,3}$ oriented in opposite directions. Now for each vertex in $T_{1,2}$ delete an edge to a uniformly chosen offspring (\ie, a neighbor of lower weight). The resulting configuration $\go_1$ consists only of infinite rays. We construct a percolation configuration $\omega$ on $\DL(2,3)$ by
    \[
    \left((x,y),(x',y')\right)\in \go \Leftrightarrow (x,x')\in \go_1 \textnormal{ and } (y,y')\in E(T_2).
    \]
    Clearly, such construction is invariant. Notice that every cluster in $\go$ is a copy of a canopy tree, in particular, it is heavy and one-ended. The weights along the geodesic path towards this end decay exponentially. 
\end{ex}

%
\section{Percolation phase transitions}\label{sec:phasetrans}
%

\subsection{The question of $p_h<p_u$}
We begin by recalling a result from \cite{Timar06nonu} rephrased to our terminology.

\begin{thm}[{\cite[Corollary 5.8]{Timar06nonu}}]\label{cor: Adam5.8}
  Let $\Gamma$ be a closed subgroup of $\Aut(G)$ that acts transitively on $G$ and $\w$ be the induced relative weight function as in \eqref{def:haarweights}. Then if $G$ is level-amenable then there cannot be infinitely many heavy clusters in Bernoulli$(p)$ percolation on $G$ for any $p\in[0,1]$.
\end{thm}

A conjecture due to Hutchcroft \cite[Conjecture 8.5]{Hutchcroft20} states that this is, in fact, a characterization of the existence of the infinite phase for the heavy clusters. In light of Theorem~\ref{thm:levelamen-hf}, this conjecture can be restated as a generalization of the ``$p_c<p_u$'' conjecture of Benjamini and Schramm mentioned above, \cite[Conjecture 6]{BSbeyond}.

\begin{conj}\label{conj:ph_vs_pu}
    Let $\Gamma$ be a closed subgroup of $\Aut(G)$ that acts transitively on $G$ and $\w$ be the induced relative weight function as in \eqref{def:haarweights}. Then $G$ is $\w$-amenable if and only if $p_h(G,\Gamma)=p_u(G)$.
\end{conj}
Notice that indeed, when $G$ is unimodular, $\w$-amenability reduces to the regular notion of amenability for graphs and $p_c(G)=p_h(G,\Gamma)$ recovering \cite[Conjecture 6]{BSbeyond}.

We obtain the following analog of Gandolfi--Keane--Newman result \cite{pc=pu} about $p_c=p_u$ for unimodular amenable graphs. We will present three ways of deriving this corollary.
\begin{cor}\label{thm:ph=pu}
    Let $\Gamma$ be a closed subgroup of $\Aut(G)$ that acts transitively on $G$ and $\w$ be the induced relative weight function as in \eqref{def:haarweights}. If $G$ is $\w$-amenable then $p_h(G,\Gamma)=p_u(G)$.
\end{cor}

\begin{proof}[First proof of Corollary~\ref{thm:ph=pu}] By Theorem ~\ref{thm:levelamen-hf} we know that $\w$-amenability implies level-amenability, so Theorem~\ref{cor: Adam5.8} implies the claim.
\end{proof}

\begin{proof}[Second proof of Corollary~\ref{thm:ph=pu}]
It follows from the relative Burton--Keane theorem \cite[Theorem 1.7]{HutchcroftPan24rel} of Hutchcroft and Pan.
\end{proof}

\begin{proof}[Third proof of Corollary~\ref{thm:ph=pu}]
    Suppose that $p_h<p_u$. Then for any $p\in(p_h,p_u)$ there are infinitely many heavy clusters. By Theorem~\ref{thm:CorBernPerc} for $\pr_p$-a.e.~configuration $\go$, the random forest $\fmax_{\w}(\go)$ almost surely contains infinitely many trees with infinitely many $\w$-nonvanishing ends. Thus by Theorem~\ref{thm:cluster-graphing_nonamenable} the cluster graphing of such a forest is $\mu$-nonamenable. Finally, by Observation~~\ref{obs:cluster}\ref{obs:cluster-amen} the group $\Gamma$ also must be nonamenable. Theorem~\ref{thm:BLPSw-amen} concludes the proof.
\end{proof}

Let us also mention that in \cite[Corollary 4.8]{Pengfei18} $p_h(G,\Gamma)\neq p_u(G)$ is characterized by the existence of an increasing exhausting sequence of slices $\{G_n\}_{n\in\bN}$ such that $\lim_{n\to\infty}p_c(G_n)< p_u(G).$

\begin{lem}\label{lem:nonamenlvl}
    Let $\Gamma$ be a closed subgroup of $\Aut(G)$ that acts transitively on $G$ and $\w$ be the induced relative weight function as in \eqref{def:haarweights}. If $G$ is $\w$-nonamenable then there exists a transitive graph $G'$ on the same vertex set such that
    \begin{enumerate}
        \item\label{lem:nonamenlvl1} $G'$ is quasi-isometric to $G$.
        \item\label{lem:nonamenlvl2} $\Gamma$ is a closed subgroup of $\Aut(G')$ and induces the same relative weight function $\w$.
        \item\label{lem:nonamenlvl3} Every level of $G'$ induces a nonamenable graph.
    \end{enumerate}
\end{lem}
\begin{proof}
    Consider a $k>1$ and a graph $L_k$ with some unimodular group $\Delta$ of automorphisms acting on it in a quasi-transitive way with $k$ orbits. Let $V_{k-1}$ be an arbitrary union of $k-1$ of these orbits.  We will prove that there is a graph $L_{k-1}$ on $V_{k-1}$, such that every element of $\Delta|_{V_{k-1}}$ defines an isomorphism of $L_{k-1}$, and such that $L_{k-1}$ is quasi-isometric to $L_k$. 
    Moreover, the edges of $L_{k-1}$ will be defined as a local function of edges of $L_k$, (i.e., the $L_{k-1}$-edges on a vertex are given by a function from a suitable ball around the vertex, and the function is invariant under the rooted isomorphisms of this ball). 
    Once we have this construction, the proof is completed as follows. First use Theorem \ref{thm:levelamen-hf} to find a nonamenable subgraph $L_k$ of $G$ that is induced by the union of $k$ levels for a suitable $k<\infty$. Remove the level of lowest weight from $L_k$, call the remaining vertices $V_{k-1}$, and define $L_{k-1}$ for $L_k$ as above. Repeat this to obtain $L_{k-2}$ from $L_{k-1}$, and so on until we arrive to an $L_1$ where the restriction of $\Gamma$ to $L_1$ is transitive, and $L_1$ is quasi-isometric to $L_2$ and (inductively) to $L_k$. In particular, $L_1$ in nonamenable because $L_k$ is. Now, to every level of $G$ add edges as in $L_1$. The resulting graph $G'$ satisfies all the requirements of the theorem.

    To define $L_{k-1}$ from $L_k$, let the edge set of $L_{k-1}$ be the union of the edges of $L_k$ restricted to $V_{k-1}$, and all pairs of the form $\{u,v\}$ where $u$ and $v$ are in the same orbit and there are vertices $u',v'\in V(L_k)\setminus V_{k-1}$ such that $u$ is adjacent to $u'$, $v$ is adjacent to $v'$, and either $u'=v'$ or $u'$ and $v'$ are adjacent. To see that $L_{k-1}$ is quasi-isometric to $L_k$, note that every path in $L_k$ can be replaced by a path in $L_{k-1}$ whose length is at most twice its length.
\end{proof}

We now present the proof for the relaxation of Hutchcroft's conjecture.
\begin{proof}[Proof of Theorem~\ref{thm:wPSN}]
        By Lemma~\ref{lem:nonamenlvl} we may assume that each level of $G$ induces a nonamenable graph. Fix $k\in\bN$ and construct a multi-graph $G^{(k)}$ from $G$ by placing an edge between two endpoints of every path (in $G$) of length $\le k$ that starts and ends in the same level (with multiplicities). By \cite[Lemma 2]{PSN00} we have that 
    \begin{equation}\label{tendstoone}
        \frac{\Phi_E(L^{(k)})}{d(L^{(k)})}\to 1 \textnormal{ as } k\to\infty,
    \end{equation}
        where $L^{(k)}$ is the subgraph induced by a level in $G^{(k)}$, and $d(L^{(k)})$ denotes the degree in it.
        
        For any finite subgraph $F^{(k)}\subseteq G^{(k)}$ rewrite $\partial_E F^{(k)}$ as a disjoint union of the external edge boundary within levels and between levels, \ie\ $\partial_E F^{(k)}=E_{F^{(k)}}\sqcup \partial_E F^{(k)}\setminus E_{F^{(k)}}$, where
    \begin{align*}
        E_{F^{(k)}}:=\{(x,y)\in \partial_E F \mid  y\in\mathrm{Level}(x)\}.
    \end{align*}
    Notice that the size of $\partial_E F^{(k)}\setminus E_{F^{(k)}}$ is constant in $k$, while the size of $E_{F^{(k)}}$ increases as $k\to\infty$.
    Thus
    \begin{equation}\label{closeness}
    \frac{d(L^{(k)})}{d(G^{(k)})} \to 1 \quad\textnormal{as } k\to\infty.
    \end{equation}
    
   Consider any finite set $F\subseteq G^{(k)}$, and let $L_1,L_2,\ldots,L_m$ be the levels that it intersects. Then
    \begin{equation*}
        \frac{\abs{\partial_{E(G^{(k)})}F}}{\abs{F}}\ge\frac{\sum_{i=1}^m\abs{\partial_{E(L_i)}(F\cap L_i)}}{\sum_{i=1}^m\abs{F\cap L_i}}\ge \min_{i\in[m]}\frac{\abs{\partial_{E(L_i)}(F\cap L_i)}}{\abs{F\cap L_i}}\ge \Phi_E(L^{(k)}),
    \end{equation*}
    where the first inequality holds as we discarded edges on the boundaries that connect different levels, and the last inequality holds because all levels are isomorphic to each other. Taking infimum over all $F$ yields that 
    \begin{equation}\label{eq:cheegeq}
        \Phi_E(G^{(k)})\ge\Phi_E(L^{(k)}) \quad\textnormal{ for all }k\in\bN.
    \end{equation}
    
    By this, \eqref{tendstoone} and \eqref{closeness}, if $k$ is large enough then $\frac{\Phi_E(L^{(k)})}{d(L^{(k)})}> \frac{1}{\sqrt{2}}$ and even $\frac{\Phi_E(L^{(k)})}{d(G^{(k)})}> \frac{1}{\sqrt{2}}$. Thus, by \eqref{eq:cheegeq}, there exists $k$ such that $\frac{\Phi_E(G^{(k)})}{d(G^{(k)})}>\frac{1}{\sqrt{2}}$ and hence by \cite[Theorem 7.38]{LyonsBook} we have
    $\rho(G^{(k)})<\frac{1}{\sqrt{2}}$.
    Then we derive that
    \begin{align*}
    p_h\left(G^{(k)},\Gamma\right)&\le p_c(L^{(k)})\le \frac{1}{\Phi_E(L^{(k)})+1}<\frac{\sqrt{2}}{d(G^{(k)})}<\frac{1}{\rho(G^{(k)})d(G^{(k)})}
    \le p_u(G^{(k)}),
    \end{align*}
    where the second and the last inequalities are from Benjamini and Schramm \cite{BSbeyond}, see \cite[Theorem 6.46]{LyonsBook} and \cite[Theorem 7.32, Lemma 7.33]{LyonsBook} respectively. Choosing $G':=G^{(k)}$ completes the proof.
\end{proof}
Let us mention that by using \cite{ThomPSN} instead of \cite{PSN00} in the above proof one can attain a graph $G'$ in the theorem that has no parallel edges. 

In light of Theorem~\ref{thm:wPSN} one can ask the following equivalent of Conjecture~\ref{conj:ph_vs_pu}.

\begin{ques}
Let $G$ and $G'$ be graphs on the same vertex set and $\Gamma$ be a closed subgroup of their automorphism groups that acts transitively on both. Does $p_h(G)<p_u(G)$ imply $p_h(G')<p_u(G')$? More generally, is $p_h<p_u$ preserved by quasi-isometries that do not change the weights?
\end{ques}

\subsection{Continuity of the percolation phase transition for $\w$-nonamenable graphs}

In this section, we prove that there are no heavy clusters under critical Bernoulli$(p_h)$ bond percolation on a $\w$-nonamenable graph $G$. In the unimodular setting it was shown in \cite{BLPSpc}. The structure of our proof resembles the original argument (as in \cite[Theorem 8.21]{LyonsBook}), but is presented in terms of nonhyperfiniteness of the heavy clusters. Our contribution is in the development of proper weighted analogs of the statements referenced in the argument.

It is easy to show that in Theorem~\ref{thm:BLPSnew} one can replace \eqref{thm:BLPSnew3} by ``there is a $\Gamma$-invariant random connected hyperfinite subgraph of $G$''. This modification leads to the following result.
\begin{thm}\label{thm:heavy_nonhf}
    Let $\Gamma$ be a closed subgroup of $\Aut(G)$ that acts transitively on $G$ and $\w$ be the induced relative weight function as in \eqref{def:haarweights}. Suppose $G$ is $\w$-nonamenable.
    Suppose $p\in[0,1]$ is such that there is a heavy cluster $\pr_p$-a.s. Then $\pr_p$-a.s.,\ every heavy cluster in $\omega$ is not hyperfinite, \ie\ $\pr_p(C(o) \textnormal{ is not hyperfinite}\mid \w(C(o))=\infty)=1.$
\end{thm}
\begin{rem}
    Theorem~\ref{thm:heavy_nonhf} holds more generally for any insertion and deletion tolerant $\Gamma$-invariant percolation that admits a heavy cluster.
\end{rem}
\begin{proof}[Proof of Theorem~\ref{thm:heavy_nonhf}]
    In the case when there is a unique heavy cluster, one can extend a hyperfinite exhaustion of $\go$ to an exhaustion of $G$ using the same argument as in  Theorem~\ref{thm:BLPSnew}\eqref{thm:BLPSnew3}. 
    
    In the case when there are infinitely many heavy clusters, Theorem~\ref{thm:CorBernPerc} implies that a.s.,~for every heavy cluster $C\subseteq\go$, there is a tree $T\subseteq\fmax_\w(C)$ with $\ge3$ $\w$-nonvanishing ends. So with positive probability $o\in T$ and on this event $T$ is not hyperfinite by Theorem~\ref{thm:trees and ph}. Hence $C$ is also not hyperfinite.
\end{proof}
This immediately implies Theorem~\ref{thm:contph}.
\begin{proof}[Proof of Theorem~\ref{thm:contph}]
    The first part of the statement follows from the fact that percolation configuration at $p_h$ is by definition hyperfinite, hence by Theorem~\ref{thm:heavy_nonhf} it contains only light clusters.
    
    The inequality $p_h(G,\Gamma)<1$ follows from the first part or even more directly from Theorem~\ref{thm:Ewh{D}inv}.
\end{proof}

We now give an alternative proof of one of the main results in \cite{Timar06nonu}, which implies \cite[Corollary 5.8]{Timar06nonu} (as in Theorem~\ref{cor: Adam5.8}) used in the previous subsection. In fact, we present a slightly stronger (but conjecturally equivalent) statement as we do not require the existence of infinitely many heavy clusters. Instead, we assume that $G$ is $\w$-nonamenable. 

\begin{cor}[Slight strengthening of {\cite[Theorem 5.5]{Timar06nonu}}]
    Let $\Gamma$ be a closed subgroup of $\Aut(G)$ that acts transitively on $G$ and $\w$ be the induced relative weight function as in \eqref{def:haarweights}.
    Suppose $G$ is $\w$-nonamenable and $p\in[0,1]$ is such that there is a heavy cluster $\pr_p$-a.s. Then a.s.~for any infinite cluster $C$ there exists a slice $S$ such that $C\cap S$ induces an infinite connected component.
\end{cor}
\begin{proof}
    Suppose there is a heavy cluster $C\subset\go$ such that for any slice $S$ we have that $C\cap S$ induces only finite components. Consider level-percolation of $G$ defined by a hyperfinite exhaustion of the level graph $\cL(G)$. Since each connected component of such percolation is a slice, intersecting it with $\go$ yields a hyperfinite exhaustion of $C$, contradicting Theorem~\ref{thm:heavy_nonhf}.
\end{proof}

Tracing down all of the statements on which this proof relies both in the present work and in \cite{FmaxSF22} one can see that it is indeed independent of \cite{Timar06nonu}. However, some of the ideas in our argument are similar to those in \cite{Timar06nonu}. For example, both approaches rely on partitioning of the levels, either by taking a 1-partition or a level-percolation.
The proof in \cite{Timar06nonu} was based on a construction of an infinitely-ended forest on a finite set of levels, thus reducing the problem to unimodular quasi-transitive graphs (similarly to the present proof of Theorem \ref{thm:wPSN}). Here, through the use of the $\fmax$ with trees of infinitely many nonvanishing ends, we get a direct contradiction.

%
\section{Questions and further directions}\label{sec:next}
%

\subsection{Upper bounds on $p_h$}\label{subsec:ph}
The inequality of Benjamini and Schramm $p_c(G)\le 1/(\Phi_V(G)+1)$ \cite[Theorem 2]{BSbeyond} raises the following question.
\begin{ques}\label{ques:ph}
    Let $\Gamma$ be a closed subgroup of $\Aut(G)$ that acts transitively on $G$ and $\w$ be the induced relative weight function as in \eqref{def:haarweights}. Does the following inequality hold 
    \[
    p_h(G,\Gamma)\le \frac{1}{\Phi^{\w}_V(G)+1}?
    \]
\end{ques}
An identical argument to the that of the fact that $p_c^{\mathrm{bond}}(G)\le p_c^{\mathrm{site}}(G)$ \cite{Hammersley61} as presented in \cite[Proposition 7.10]{LyonsBook} yields the following.
\begin{lem}\label{lem:bond<site}
    Let $\Gamma$ be a closed subgroup of $\Aut(G)$ that acts transitively on $G$, then
    \[
    p_h^{\mathrm{bond}}(G,\Gamma)\le p_h^{\mathrm{site}}(G,\Gamma).
    \]
\end{lem}
\noindent
Thus Theorem~\ref{thm:site-threshold} and Lemma~\ref{lem:bond<site} imply 
    \begin{equation}\label{eq:uppbdphwithd}
    p_h^{\mathrm{bond}}(G,\Gamma)\le p_h^{\mathrm{site}}(G,\Gamma)\le \frac{d}{\Phi^{\w}_V(G)+d},
    \end{equation}    
where $d$ is the degree in $G$. 

For nonamenable graphs the question of $p_c<1$ was known from \cite{BSbeyond},
and later it was solved for all transitive graphs of nonlinear growth by Duminil-Copin, Goswami, Raoufi, Severo, and Yadin \cite{Hugo20}, answering \cite[Conjecture 2]{BSbeyond}.
In analogy, the second part of our Theorem~\ref{thm:contph} shows $p_h<1$ for $\w$-nonamenable graphs, and one may wonder how far the conditions can be relaxed.  

\begin{ques}\label{upward_growth}
Let $G$ be a graph with a closed subgroup of $\Aut(G)$ that acts transitively on it. Find a necessary and sufficient condition for $p_h<1$.
\end{ques}

When $H$ is the upper half space of $G$ (that is, the subgraph induced by vertices of weights at least 1) and $H$ is uniformly transient, $p_h(G)\leq p_c(H)<1$ follows from the theorem of Easo, Severo and Tassion \cite{easo2024pc}, as pointed out to us by Tom Hutchcroft. However, the Cartesian product of a grandparent graph and $\bZ$ is an example where uniform transience does not hold for $H$, but we still have $p_h(G)<1$ (this follows from the fact that $p_u(G)<1$ by \cite[Theorem 4.1.7]{Haggstrom99}).

In the nonunimodular setting one can consider an \textbf{upward growth} by counting the number of vertices $x\in B_n(o)$ such that there is a path between $o$ and $x$ of nondecreasing weights. Is nonlinear upward growth equivalent to $p_h<1$?

\subsection{Weighted random spanning trees and forests}
In \cite{FmaxSF22} the authors introduced $\w$-$\fmax$ as a modification of $\FMSF$, which is handy in the nonunimodular setting as it inherently respects $\w$.
Indeed, in Theorem~\ref{thm:CorBernPerc} and in the proof of Theorem~\ref{thm:contph}, $\fmax_{\w}$ played analogous role to the one of $\FMSF$ in the proof of \cite[Theorem 8.21]{LyonsBook}.
There are several open questions about this forest. Firstly, \cite[Questions 1.13 and 1.14]{FmaxSF22} concern indistinguishability of the $\w$-Free Maximal Spanning Forest, asking to generalize \cite{timar2006ends,Timar18ind}. 

Define the Wired Maximal Spanning Forest $\mathrm{WMaxSF}_{\w}$ by considering every bi-infinite path as a cycle. Which properties of the Wired Minimal Spanning Forest do extend, in particular, does the next analogue of \cite[Proposition 3.6]{LPS06msf} hold? 

\begin{ques}\label{ques:ph}
    Let $\Gamma$ be a closed subgroup of $\Aut(G)$ that acts transitively on $G$.  Is it true that $p_h(G,\Gamma)=p_u(G)$ if and only if $\fmax_{\w}(G)$=$\mathrm{WMaxSF}_{\w}(G)$?
\end{ques}

\subsection{Modular random graphs} In a follow-up work together with G\'abor Elek, we introduce {\it (multi)modular random graphs} as a generalization of nonunimodular quasi-transitive graphs, in the spirit of the unimodular random graphs of Aldous and Lyons \cite{URG}. Most arguments used in the present paper rely on the tilted mass transport principle, and could therefore be extended to this wider setting.

\subsection*{Acknowledgments}
We thank Russell Lyons, Robin Tucker-Drob, Konrad Wr\'obel and Tom Hutchcroft for helpful discussions, and Gabor Elek for suggesting Example~\ref{ex:freeprodtower}. We also thank Anush Tserunyan for carefully reading the original draft and for many helpful suggestions that improved the article. The first author was supported in part by the RTG award grant (DMS-2134107) from the NSF, the second author by the Icelandic Research Fund grant No. 239736-051. Both authors were also partially supported by the ERC Synergy Grant No. 810115 - DYNASNET.

\bibliographystyle{alpha}
\bibliography{noninimod} 
\end{document}